\documentclass[11pt]{article}
\usepackage{delarray}
\usepackage{amsmath, latexsym, amsfonts, amssymb, amsthm, amscd}
\usepackage{graphicx}
\usepackage{graphicx,,psfrag}
\usepackage{amssymb,amsmath,amsthm,enumerate}
\usepackage{color}
\newcommand{\E}{\mathbb{E}}
\newcommand{\C}{\mathbb{C}}
\newcommand{\R}{\mathbb{R}}
\newcommand{\tr}{\operatorname{tr}}
\newcommand{\Tr}{\operatorname{Tr}}
\newcommand{\vers}{\mathop{\longrightarrow}} 
\newcommand{\1}{1\!\!{\sf I}}

 \newtheorem{theorem}{Theorem}
 \newtheorem{proposition}{Proposition}
\newtheorem{remark}{Remark}
\newtheorem{lemma}{Lemma}
\newtheorem{corollary}{Corollary}

\begin{document}

\title{Non universality of fluctuations of outlier eigenvectors  for block diagonal deformations of Wigner matrices }
\author{M. Capitaine\thanks{\it \fontsize{8}{10}\selectfont Institut de Math\'ematiques de Toulouse; UMR5219; Universit\'e de Toulouse; CNRS;
 UPS, 118 rte de Narbonne F-31062 Toulouse, FRANCE. 
E-mail: mireille.capitaine@math.univ-toulouse.fr} \  and C. Donati-Martin\thanks{\it \fontsize{8}{10}\selectfont Laboratoire de Math\'ematiques de Versailles, UVSQ, CNRS, Universit\'e Paris-Saclay, 78035-Versailles Cedex, France, E-mail: catherine.donati-martin@uvsq.fr}}
\date{}
\maketitle
\begin{abstract}
In this paper, we investigate the fluctuations of a unit  eigenvector associated to an outlier in the spectrum of a spiked $N\times N$ complex  Deformed Wigner matrix $M_N$. $M_N$ is defined as follows: $M_N =
W_N/\sqrt{N} + A_N$ where $W_N$ is an $N \times N$ Hermitian  Wigner matrix whose entries have a  law $\mu$ satisfying
a Poincar\'e inequality and the matrix $A_N$ is a block diagonal matrix, with an eigenvalue $\theta$ of multiplicity one,  generating an outlier in the spectrum of $M_N$. We  prove that  the fluctuations of the norm of the projection of a unit  eigenvector   corresponding to the outlier of $M_N$ onto a unit  eigenvector  corresponding to $\theta$ are not universal.
Indeed, we take away a fit approximation of its limit from this norm   and prove the convergence to zero as $N$ goes to $\infty$
of  the
L\'evy–Prohorov distance between this  rescaled quantity and 
  the convolution of $\mu$  and a centered Gaussian distribution (whose variance may depend 
depend upon N  and may not
converge).\\

\noindent {\it Key words:} Random matrices; Outliers; eigenvectors; Fluctuations; Nonuniversality, Free probability.\\

\noindent Mathematics Subject Classification 2000: 15A18, 15A52, 60F05, 60B12, 46L54.
\end{abstract}
\section{Introduction}
To begin with, we introduce
some notations.
\begin{itemize}
\item $M_N(\C)$ is the set of $N\times N$ matrices  with complex entries,  $M_N^{sa}(\C)$ the subset of self-adjoint elements of $M_N(\C)$ and $I_N$ the identity matrix. 
\item $\Tr_N$ denotes the trace and $\tr_N = \frac{1}{N} \Tr_N$ the normalized trace on $M_N(\C)$.
\item $|| . ||$ denotes the operator norm on $M_N(\C)$.
\item For any $X \in M_N^{sa}(\C)$, $(\lambda_1(X), \ldots,\lambda_N(X))$ denote the eigenvalues of $X$ ranked in decreasing order and the empirical spectral measure of $X$ is defined by $$\mu_{X}:=\frac{1}{N} \sum_{i=1}^N \delta_{\lambda_i(X)}.$$
\item For a probability measure $\tau$ on $\R$, 
${\rm supp}(\tau)$ denotes the support of $\tau$ and  $g_\tau: z\in \mathbb{C}\setminus {\rm supp}(\tau) \mapsto \int\frac{1}{z-x} d\tau(x)$ is the Stieltjes transform of $\tau$.
\item $d_{LP}$ denotes the L\'evy-Prohorov distance, which is a metric for the topology of the convergence in distribution.
\end{itemize}
\subsection{Wigner matrices}
 Wigner matrices are  complex Hermitian random matrices whose entries are independent (up to the symmetry condition).
They were introduced by Wigner in the fifties, in connection with nuclear physics.
Here, we will consider Hermitian Wigner matrices of the following form :
$$X_N= \frac{1}{\sqrt{N}} W_N$$
where $W_N$ is an Hermitian matrix, $\{W_{ii}, \sqrt{2}\mathcal{R}W_{ij}, \sqrt{2}\mathcal{I}W_{ij}\}_{1\leq i< j}$ 
are independent identically distributed  random variables with law $\mu$, with mean zero and  variance $\sigma^2$.
  If the entries are independent  Gaussian variables, $ X_N=:X_N^G$ is a matrix from the Gaussian Unitary Ensemble (G.U.E.).\\
 There is currently a quite precise knowledge of the asymptotic  spectral properties
 (i.e. when the dimension of the matrix tends to infinity) of Wigner matrices. {This understanding covers both the so-called
 global regime (asymptotic behavior of the spectral measure) and the local regime
(asymptotic behavior of the extreme eigenvalues and  eigenvectors, spacings...).} 
Wigner  proved that a precise description of the limiting spectrum of these matrices
can be achieved.  
  \begin{theorem}\label{wig}\cite{Wigner55,Wigner58} 
 $$\mu_{X_N}\overset{w}{\longrightarrow} \mu_{sc}  \mbox{~~a.s. when~~} N \rightarrow + \infty  $$
 where 
\begin{equation}\label{sc}  \frac{d \mu_{sc}}{dx}(x)=  \frac{1}{2 \pi \sigma^2}
\sqrt{4\sigma^2 - x^2} \, 1_{[-2\sigma,2 \sigma]}(x)\end{equation}
is the so-called semi-circular distribution.
\end{theorem}

A priori, the convergence of the spectral measure does not prevent an asymptotically negligeable fraction of eigenvalues from going away from the limiting support (called {\it outliers} in the following).
Actually, it turns out that Wigner matrices  do not  exhibit outliers. 

\begin{theorem}\label{baiyin} \cite{BaiYin88}   
 Assume that the entries of $W_N$ has finite fourth moment, then almost surely, $$\lambda_{1}(X_N) \rightarrow 2 \sigma
\mbox{~and~} \lambda_{N}(X_N) \rightarrow -2 \sigma
\mbox{~ when~} N \rightarrow + \infty.$$
\end{theorem}

 In \cite{TW}, Tracy and Widom
derived the limiting distribution (called   the Tracy-Widom law) of the largest eigenvalue of a G.U.E. matrix. 
\begin{theorem}
 Let $q: \mathbb{R}\rightarrow \mathbb{R} $ be the unique solution of the differential equation
$$q''(x)= x q(x) +2 q(x)^3$$ such that $q(x) \sim_{x\rightarrow +\infty} Ai(x)$ where $Ai$ is the Airy function, unique solution on $\mathbb{R}$ of the differential equation $f''(x)=xf(x)$ satisfying $f(x)\sim_{x\rightarrow +\infty} (4\pi \sqrt{x})^{1/2} \exp(-2/3x^{3/2})$.
Then $$\lim_{N \rightarrow +\infty} \mathbb{P} \left(\frac{N^{2/3}}{\sigma} \left( \frac{\lambda_1(X_N^G)}{\sqrt{N}}-2\sigma\right) \leq s \right) =F_2(s),$$
where $F_2(s)=\exp \left( -\int_{s}^{+\infty} (x-s) q^2(x) dx\right).$

\end{theorem}
 The first main step to prove the universality conjecture for fluctuations of the largest eigenvalue of Wigner matrices  has been achieved by  Soshnikov \cite{Sos99};
in \cite{LeeYin}, a necessary and sufficient  condition  on  off-diagonal entries of the Wigner matrix  is established 
 for the distribution  of the  largest eigenvalue    to weakly converge to  the Tracy-Widom distribution. We also refer to these  papers  for references on  investigations on edge universality.
 
In regards to eigenvectors, it is well known that the matrix whose columns are the   eigenvectors of
a  G.U.E. matrix  can be chosen to be distributed according to the Haar measure
on the unitary group. In the non-Gaussian case, the exact distribution of the eigenvectors cannot be computed. However,
 the eigenvectors of general Wigner matrices   have been the object
of a growing interest and in several papers, a delocalization and universality property
were shown for the eigenvectors of these  standard models (see among others \cite{BE,  E, E2, KYev, TV3} and references therein). Heuristically, delocalization for a random matrix means that its normalized eigenvectors look like the vectors uniformly distributed over the unit sphere.
Let us state for instance the following sample  result.
\begin{theorem} (Isotropic delocalization, Theorem 2.16 from \cite{BE}). 
Let $X_N$ be a $ N\times N$ Wigner matrix  satisfying some technical assumptions.  Let $v(1), \ldots, v(N)$ denote the normalized eigenvectors of $X_N$. Then, for any $C_1 > 0$
and $0<\epsilon< 1/2$, there exists $C_2 > 0$  such that
$$\sup_{
1\leq i\leq N} |\langle v(i) , u \rangle| \leq  \frac{N^\epsilon}{\sqrt{N}},$$
for any fixed unit vector $u \in  \mathbb{C}^N$, with probability at least $ 1- C_2 N^{-C_1}$.
\end{theorem}
\subsection{Deformed Wigner matrices}
Practical problems (in the theory of statistical learning, signal detection
etc.) naturally lead to wonder about the spectrum reaction
of a given random matrix after a deterministic perturbation. In those applications, the random matrix is the noise and the perturbed matrix is a noisy version of the information; the question  is to know whether the observation of the spectral properties 
 of the perturbed matrix can
give access to  significant parameters on the information. 
Theoretical results on these deformed random models may allow  to establish
statistical tests on these parameters. {A typical illustration is the so-called BBP phenomenon
(after Baik, Ben Arous, P\'ech\'e) which put forward outliers
(eigenvalues that move away from the rest of the spectrum) and their Gaussian fluctuations
for spiked covariance matrices in \cite{BBP} and for low rank deformations of G.U.E. in \cite{Peche}.}\\
In this paper, we consider additive perturbations of Wigner matrices.
The pionner works on additive deformations go back to Pastur \cite{Pa} for the behavior of the limiting spectral distribution and to F{\"u}redi and Koml{\'o}s \cite{FK} for the behavior of the largest eigenvalue. \\
We refer to \cite{CD} and the references therein for a survey on spectral properties of deformed random matrices.  \\
The model studied is as follows : 
\begin{equation}\label{defMN}M_N:=\frac{W_N}{\sqrt{N}}+A_N,\end{equation}
 where

~~
 \\
 {\bf (W)}  $W_N$ is a complex Wigner matrix, 
that is a $N\times N$ random Hermitian matrix such that $\{W_{ii}, \sqrt{2}\mathcal{R}W_{ij}, \sqrt{2}\mathcal{I}W_{ij}\}_{1\leq i< j}$ 
are independent identically distributed  random variables with law $\mu$. 
We assume that $\mu$ is a distribution with mean zero, variance $\sigma^2$,  and  satisfies a Poincar\'e inequality (see Appendix).
Note that this condition implies that $\mu$ has moments of any
order (see Corollary 3.2 and Proposition 1.10 in \cite{L}).

~~
\\
{\bf (A)} $A_N$ is a $N\times N$ deterministic Hermitian matrix, whose spectral measure $\mu_{A_N}$
converges to  a compactly supported probability measure $\nu$.
We assume that $A_N$ has a fixed number $q$ of eigenvalues, not depending on $N$, outside the support of $\nu$ called spikes, whereas the distance of the other eigenvalues to the support of $\nu$ goes to 0. 

\medskip
\noindent
The empirical spectral distribution $\mu_{M_N}$ converges a.s. towards the probability measure $\lambda := \mu_{sc} 
\boxplus \nu$ where $\mu_{sc} $ is the semicircular distribution with variance $\sigma^2$ and $\boxplus$ denotes the free convolution, see \cite{Pa} (in this paper, the limiting distribution is given via a functional equation for its Stieltjes transform), \cite[Theorem 5.4.5]{AGZ}. We refer to \cite{VDN, MS} for an introduction to free probability theory.\\
Concerning extremal eigenvalues, 
\cite{CDFF} proved  that the spikes of $A_N$ can generate outliers for the limiting spectrum of $M_N$, i.e. eigenvalues outside the support of the limiting distribution $\lambda$. More precisely, \cite{CDFF} proved the following (see \cite[Theorem 8.1]{CDFF} for a more general statement).
\begin{proposition} \label{propCVoutlier}\cite{CDFF}
Assume that a spike $\theta$   with a fixed  multiplicity $k_0$ in the spectrum of $A_N$ satisfies :
\begin{equation} \label{cond-outlier}
 \theta \in \Theta_{\sigma,\nu}:=\{ u \in \mathbb R \backslash {\rm supp}(\nu), \int_\mathbb R \frac{d\nu(x)}{ (u-x)^2} < \frac{1}{\sigma^2} \}.
 \end{equation}
 Denote by $n_{0}+1, \ldots , n_{0}+k_0$ the descending ranks of $\theta$ among the eigenvalues of $A_N$.
Then the $k_0$ eigenvalues $(\lambda_{n_{0}+i}(M_N), \, 1 \leq i \leq k_0)$ 
converge almost surely outside the support of $\lambda$
towards $\rho _{\theta} := \theta + \sigma^2 g_\nu(\theta)$.
Moreover, these eigenvalues asymptotically separate from the rest of the spectrum since 
(with the conventions that $\lambda_0(M_N)=+\infty$ and $\lambda_{N+1}(M_N)=-\infty$)
there exists  $0< \delta_0 $ such that 
\noindent almost surely for all large N, \begin{equation}\label{sepraj}\lambda_{n_{0}}(M_N) > \rho _{\theta} + \delta_0 \, \mbox{~and~} \, \lambda_{n_{0}+k_0 +1}(M_N) < \rho _{\theta} - \delta_0.\end{equation}
\end{proposition}
Note that \cite{CDFF} assumes that the distribution $\mu$ is symmetric but this assumption can be removed. Indeed, this assumption is used for establishing  Theorem 5.1   in \cite{CDFF} that is now  generalized by Theorem 1.1 in  \cite{BC}. Note also that \cite{CDFF} assumes moreover  that the support of  $\nu$ has a finite number of connected components in order to prove Theorem 6.1 in \cite{CDFF} but this assumption is removed in Theorem 2.3 in \cite{CP}. 
\begin{remark}
Note that  
$$ \{ u \in \mathbb R \backslash {\rm supp}(\nu), \int_\mathbb R \frac{d\nu(x)}{ (u-x)^2} < \frac{1}{\sigma^2} \}= \{ u \in \mathbb R \backslash {\rm supp}(\nu), H'(u)>0\}$$ and for any $\theta$ in this set, 
$\rho_\theta =  \theta + \sigma^2 g_\nu(\theta)=H(\theta)$,  where
$H$ is defined by \eqref{H} (see Section \ref{freeconv} below).
\end{remark}
It turns out that we can also describe the angle between the eigenvector associated to the outlier of $M_N$ and the corresponding eigenvector associated to the spike $\theta$.
Capitaine \cite{C} (see also \cite{CD}) proved
\begin{proposition} \label{propCVoutlier2}\cite{C} We keep the notation and hypothesis of Proposition \ref{propCVoutlier}.
Let $\xi$ be a unit eigenvector associated to one of  the eigenvalues $(\lambda_{n_{0}+i}(M_N), \, 1 \leq i \leq k_0)$. Then, a.s. 
\begin{equation} \label{CVangle}
\Vert P_{Ker(A_N - \theta I)} (\xi) \Vert^2 \longrightarrow_{N\rightarrow +\infty} \tau(\theta) := 1 - \sigma^2 \int \frac{1}{(\theta -x)^2} d\nu(x),
\end{equation}
where $P_E$ denotes the orthogonal projection onto any subspace $E$.
\end{proposition}
Note that fluctuations of outliers for deformed non-Gaussian Wigner matrices  have been more extensively  studied  in the case of perturbations $A_N$ of fixed rank $r$.
We emphasize that the limiting distribution in the CLT for outliers depends on the localisation/delocalisation of the eigenvector of the spike. 
Roughly speaking, in the delocalized case, the limiting distribution of the fluctuations of the correponding outliers is Gaussian. In the  localized case, the limiting distribution depends on the distribution $\mu$ of the entries and thus, this uncovers a non universality phenomenon.  We refer to \cite{CDF1} for these results. \\
We first recall the fluctuations of the largest eigenvalue $\lambda_1(M_N)$ when the matrix $A_N$ is a diagonal matrix of rank 1 in the localized case. 
\begin{proposition}  \cite{CDF} \label{TCL-rangfini} Assume that $A_N = diag (\theta, 0, \ldots, 0)$ with $\theta > \sigma$.
The fluctuations of $\lambda_1(M_N)$ around $\rho_\theta =  \theta + \frac{\sigma^2}{\theta}$ are given by 
$$ c_\theta \sqrt{N} (\lambda_1(M_N) - \rho_\theta) \vers^{(law)}_{N \rightarrow \infty} \mu \star N(0, v_\theta^2)$$
where $c_\theta = (1-\frac{\sigma^2}{\theta^2})^{-1}$, $v_\theta^2 =  \frac{1}{2} \frac{m_4 - 3\sigma^4}{\theta^2} + \frac{\sigma^4}{\theta^2 - \sigma^2} $ and $m_4$ denotes the fourth moment of $\mu$.
\end{proposition}
Capitaine and P\'ech\'e \cite{CP} proved a fluctuation result for any outlier of a full rank deformation of a   G.U.E. matrix. Their result yields the following
\begin{proposition}\label{CaP}\cite{CP}
Assume that $W_N$ is a G.U.E. matrix (that is $\mu=\mathcal{N}(0,\sigma^2)$) and that $A_N$ is a diagonal matrix  with a spike $\lambda_{i_0}(A_N) = \theta \in \Theta_{\sigma, \nu}$ of multiplicity one and limiting spectral distribution $\nu$.
The fluctuations of $\lambda_{i_0}(M_N)$ around $$\rho^{(N)}_\theta = \theta + \sigma^2 \frac{1}{N-1}\sum_{\lambda_j(A_N)\neq \theta} \frac{1}{\theta-\lambda_j(A_N)}$$ 
are given by\footnote{They consider fluctuations around this point depending on $N$ in order to not prescribe  speed of convergence of $\mu_{A_N}$ to $\nu$.}:
$$
 c_{\theta ,\nu}\sqrt{N} (\lambda_1(M_N) - \rho^{(N)}_\theta) \vers^{(law)}_{N \rightarrow \infty}  N(0, \sigma_{\theta, \nu}^2)
$$
where $c_{\theta,\nu} = \left( 1 - \sigma^2 \int \frac{1}{(\theta -x)^2}d\nu(x) \right)^{-1}$ and 
\begin{eqnarray*}
\sigma_{\theta, \nu}^2  &=& \sigma^2 \left( 1 - \sigma^2  \int \frac{1}{(\theta -x)^2}d\nu(x) \right)^{-1}.
\end{eqnarray*}
\end{proposition}
\subsection{Main results}\label{main}
In the following, we consider  block diagonal perturbations.  We focus on spikes of the perturbation with multiplicity one generating an outlier in the spectrum of the   deformed Wigner model. Therefore, we shall consider the following assumption on $A_N$ throughout the paper: \\
{\it {\bf (A')} $A_N$ satisfies {\bf (A)} and 
$$A_N = {\rm diag}(A_p, A_{N-p}),$$ with $A_{N-p}$ a $(N-p) \times (N-p)$ Hermitian matrix  for some fixed integer $p$, $A_p = PDP^*$ is a fixed matrix (independent of $N$) where $P$ a  $p \times p$ unitary matrix and  $D$ is a diagonal matrix. \\
  Assume that  $A_N$   has  a spike of multiplicity one, which is an eigenvalue of $A_p$, $\theta=\lambda_{i_0}(A_N)$  for some  $i_0$, satisfying \eqref{cond-outlier}. Without loss of generality, we can assume that  $\theta = D_{11}$. } \\
\noindent We set  $$W_N=\begin{pmatrix} W_{p} & Y^*\\ Y & W_{N-p} \end{pmatrix},$$ 
where $W_{p} \in M_{p}(\mathbb{C})$, $Y\in  M_{(N-p)\times p}(\mathbb{C})$ and $W_{N-p} \in M_{N-p}(\mathbb{C})$.
\\

The main results of this paper are the following  Theorems  \ref{casnondiag} and \ref{propcasnondiag2}  on  non universality of fluctuations of an eigenvector associated with  such an outlier. But, sticking to  the   approach of \cite{CRMTA},  we first establish  
Theorem \ref{vpcasnondiag} below,
which  is   an extension in the non-Gaussian case of Proposition \ref{TCL-rangfini} and Proposition \ref{CaP}, in the block diagonal case.\\
 Let $M_N$ be defined by \eqref{defMN}   with assumptions {\bf(W)} and {\bf (A')}.
By Proposition \ref{propCVoutlier}, \begin{equation} \label{convlambda}\lambda_{i_0}(M_N)\rightarrow_{N\rightarrow +\infty} \rho_\theta \text{~a.s.}. \end{equation}
Define \begin{equation} \label{rhon} \rho_N= \theta+\sigma^2 g_{\mu_{A_{N-p}}}(\theta).\end{equation}
Note that 
\begin{equation} \label{convrho} \rho_N \rightarrow_{N\rightarrow +\infty} \rho_\theta. \end{equation}
\begin{theorem}\label{vpcasnondiag}  Let $M_N$ be defined by \eqref{defMN} with assumptions {\bf(W)} and {\bf (A')}.
Define 
\begin{equation}\label{Cp}{\bf C}_p=\left\{\begin{array}{ll} ^t com(\theta I_p -A_p), \mbox{~if~} \theta I_p -A_p \neq 0\\I_p \mbox{~else} \end{array}\right.,  \end{equation} 
where $com(B)$ denotes the comatrix of a matrix $B$, and
$$c^{(1)}_\rho= \left( 1+\sigma^2 \int \frac{d\lambda(x)}{\left( \rho_\theta -x\right)^{2}}\right)\Tr_p({\bf C}_p).$$
Then
\begin{equation}\label{cvprincipal}d_{LP}(c^{(1)}_\rho \sqrt{N} (\lambda_{i_0}(M_N)-\rho_N), \Phi_N) \vers_{N\rightarrow \infty} 0\end{equation}
where
 $\Phi_N= \Tr_p({\bf C}_p W_p) +{\cal Z}_N$, 
 $W_p$ is the $p\times p$ upper left corner of the  Wigner matrix $W_N$,  ${\cal Z}_N$ is 
  a Gaussian random variable, independent from $W_p$,    with mean 0 and variance $v_\rho(N)$,
with $$v_\rho(N)=\Tr_p({\bf C}_p^2) \sigma^4 \int \frac{d\lambda(x)}{\left( \rho_\theta -x\right)^{2}}+ \frac{1}{2} (m_4-3 \sigma^4) \kappa_N
\sum_{i=1}^p \left( ({\bf C}_p)_{ii}\right)^2,$$
$\lambda = \mu_{sc} \boxplus \nu$, $\kappa_N=\frac{1}{N-p}\sum_{i=1}^{N-p} (((\theta I_{N-p}-A_{N-p})^{-1})_{ii})^2 $.\\
In particular, if $A_{N-p}$ is diagonal,
$c^{(1)}_\rho \sqrt{N} (\lambda_{i_0}(M_N)-\rho_N)$ converges in distribution to $\Tr_p({\bf C}_p W_p) +{\cal Z}$ where 
${\cal Z}$ is 
  a Gaussian random variable, independent from $W_p$,    with mean 0 and variance $v_\rho$,
with $$v_\rho=\Tr_p({\bf C}_p^2) \sigma^4 \int \frac{d\lambda(x)}{\left( \rho_\theta -x\right)^{2}}+ \frac{1}{2} (m_4-3 \sigma^4) \int \frac{d\nu(x)}{\left( \theta-x\right)^{2}}
\sum_{i=1}^p \left( ({\bf C}_p)_{ii}\right)^2.$$
\end{theorem}
The aim of Theorems \ref{casnondiag}  and \ref{propcasnondiag2} below  is to study the fluctuations associated to the a.s. convergence given in Proposition \ref{propCVoutlier2} above, for  block diagonal perturbations. We first state an approximation result in distribution, in the spirit of \cite{NY} for perturbations $A_N$ satisfying {\bf (A')}.
\begin{theorem}\label{casnondiag}  Let $M_N$ be defined by \eqref{defMN} with assumptions {\bf(W)} and {\bf (A')}.
Let $u_{i_0}$, resp. $v_{i_0}$ be a unit eigenvector associated to the spike $\theta$ of $A_N$, resp. the outlier $\lambda_{i_0}(M_N)$.
Define $\tau_N(\theta)$  an approximation of $\tau(\theta)$  by 
\begin{equation} \label{tauN}
 \tau_N(\theta) =  1 - \sigma^2 \int \frac{1}{(\theta -x)^2} d\mu_{A_{N-p}}(x). \end{equation} 
Then
$$d_{LP}( \sqrt{N} (|\langle u_{i_0}, v_{i_0}\rangle|^2 -  \tau_N(\theta)), \Psi_N) \vers_{N\rightarrow \infty} 0$$ 
where 
 the r.v.  $\Psi_N$ is given  by 
\begin{eqnarray*} 
   \Psi_N &=& (P^*( c_{\theta, \sigma} W_p + Z_{p,N}) P)_{11},
 \end{eqnarray*}
where 
\begin{equation} \label{ctheta}
 c_{\theta, \sigma}  = \sigma^2 g''_\nu(\theta),
 \end{equation}
$W_p$ is the $p\times p$ upper left corner of the  Wigner matrix $W_N$, 
  $Z_{p,N}$ is a centered Gaussian Hermitian matrix of size $p$,  independent from $W_p$,  with independent entries (modulo the symmetry conditions). The diagonal coefficients are i.i.d. with variance 
 \begin{equation}  \label{covarianceNb}
 \sigma^4 B_{\theta, \nu} + \frac{1}{2} (m_4-3 \sigma^4) A_{\theta, \nu,N} \end{equation}
where 
\begin{equation} \label{Btheta}
B_{\theta, \nu}= -\frac{1}{6}{g_\nu'''(\theta)} -\frac{\sigma^2}{2}({g_\nu''(\theta)})^2 \frac{1+2\sigma^2 g'_\nu(\theta)}{1+\sigma^2 g'_\nu(\theta)},
\end{equation}
\begin{eqnarray}\label{defAthetaN} \lefteqn{A_{\theta, \nu, N}=} \\
&&  \frac{1}{N-p} \sum_{i=1}^{N-p} \left( \sigma^2 g''_\nu(\theta) [(\theta I_{N-p} - A_{N-p})^{-1}]_{ii} - (1+ \sigma^2 g'_\nu(\theta)) [(\theta I_{N-p} - A_{N-p})^{-2}]_{ii} \right)^2 \nonumber 
\end{eqnarray}
$m_4$ is the fourth moment of $\mu$, $g_\nu$ is the Stieltjes transform of $\nu$.
The off diagonal elements $Z_{p,N}(i,j)$, $i<j$ are iid complex Gaussian with distribution $Z$ such that $\mathbb E(Z^2) =0$ and $\mathbb E(|Z|^2) = \sigma^4 B_{\theta, \nu}$.
\end{theorem}
In the case where the matrix $A_{N-p}$ is a diagonal matrix, the sequence $(A _{\theta,\nu, N})_N$ defined in \eqref{defAthetaN} converges as $N$ tends to infinity. This leads to the following fluctuations result : 
\begin{theorem}\label{propcasnondiag2} Let $M_N$ be defined by \eqref{defMN} with assumptions {\bf(W)},  {\bf (A')} and $A_{N-p}$ diagonal.   \\
Let $u_{i_0}$, resp. $v_{i_0}$ be a unit eigenvector associated to the spike $\theta$ of $A_N$, resp.  to the outlier $\lambda_{i_0}(M_N)$. Then, 
\begin{equation}\label{Eq-vect1}
  \sqrt{N} (|\langle u_{i_0}, v_{i_0}\rangle|^2 - \tau_N(\theta)) \vers^{(law)}_{N \rightarrow \infty} \left( P^*( c_{\theta, \sigma}  W_p  + Z_p)  P\right)_{11}
 \end{equation}
where $\tau_N(\theta)$ is defined by \eqref{tauN}.
 $W_p$ is the $p\times p$ upper left corner of the  Wigner matrix $W_N$, $Z_p$ is a centered Gaussian Hermitian matrix  of size $p$, independent from $W_p$,  with independent entries (modulo the symmetry condition). The diagonal coefficients are i.i.d. with variance  
 \begin{equation}
 \frac{1}{2}(m_4 - 3 \sigma^4) A_{\theta, \nu} + \sigma^4 B_{\theta, \nu}
 \end{equation}
and the off diagonal elements are iid complex Gaussian with distribution $Z$ such that $\mathbb E(Z^2) = 0$ and $\mathbb E(|Z|^2) =  \sigma^4 B_{\theta, \nu}$
 where 
 \begin{equation} \label{cov-diag2} 
   \left\{ \begin{array}{l}
   c_{\theta,\nu}  = \sigma^2 g''_\nu(\theta), \\
 \displaystyle   A_{\theta, \nu}=  -\frac{1}{6}{g_\nu'''(\theta)}(1+\sigma^2 g_\nu'(\theta))^2-2\sigma^4 (g_\nu^{''}(\theta))^2 g_\nu'(\theta) -\sigma^2(g_\nu^{''}(\theta))^2,\\
\displaystyle B_{\theta, \nu}= -\frac{1}{6}{g_\nu'''(\theta)} -\frac{\sigma^2}{2}({g_\nu''(\theta)})^2 \frac{1+2\sigma^2 g'_\nu(\theta)}{1+\sigma^2 g'_\nu(\theta)},
\end{array} \right. \end{equation}
$m_4$ is the fourth moment of $\mu$ and $g_\nu$ is the Stieltjes transform of $\nu$.
\end{theorem}
In particular, we readily deduce 
 the following corollary.
\begin{corollary} \label{theo-vect} Let $M_N$ be defined by \eqref{defMN}  with assumptions {\bf(W)} and {\bf (A')}. Assume moreover that  $A_N$ is a diagonal matrix ($p=1$).
 Let $u_{i_0}$, resp. $v_{i_0}$ be a unit eigenvector associated to the spike $\theta$ of $A_N$, resp. the outlier $\lambda_{i_0}(M_N)$.
Then,
\begin{equation}\label{Eq-vect}
  \sqrt{N} (|\langle u_{i_0}, v_{i_0}\rangle|^2 - \tau_N(\theta)) \vers^{(law)}_{N \rightarrow \infty}  c_{\theta, \nu} W_{11}  + Z
 \end{equation}
 where $Z$ is a centered Gaussian variable, independent from $W_{11}$, with variance : 
 \begin{equation} \label{cov-diag} 
 \frac{1}{2}(m_4 - 3 \sigma^4) A_{\theta, \nu} + \sigma^4 B_{\theta, \nu}.
 \end{equation}
See Eq.\eqref{cov-diag2} for the definitions of $c_{\theta, \sigma}$, $A_{\theta, \nu}$, $B_{\theta, \nu}$.
   \end{corollary}
Note that when the Wigner matrix is Gaussian, the choice of a diagonal matrix for $A_N$ is not a restriction, due to the unitary invariance of the G.U.E..\\


 The proof of  Theorem  \ref{casnondiag} 
 relies  upon a representation, through Helffer-Sj\"{o}strand formula, of the variable $|\langle u_{i_0}, v_{i_0}\rangle|^2$ in terms  of the $p\times p$-matrix valued process
$ \{ G_p(z), z \in \mathbb C \backslash \mathbb R \}$
where $G_p(z)$ denotes the principal submatrix of size $p$ of the resolvant matrix $G(z) = (zI_N-M_N)^{-1}$. Then, the fluctuations of the process $ \{ G_p(z), z \in \mathbb C \backslash \mathbb R \}$ are analysed using Schur's formula which enables to express $ \{ G_p(z), z \in \mathbb C \backslash \mathbb R \}$ in terms of random sesquilinear forms. This approach is described in Section \ref{common} where we also prove Theorem \ref{propcasnondiag2}. Section \ref{prelim}  gathers preliminary results used in the proof of the main results. Therein, first we recall classical algebraic identities, Helffer-Sj\"{o}strand's calculus, tightness criterion for  random analytic process and some basic facts on free convolution with a semicircular distribution; later we establish some extension of central limit theorem for random quadratic forms and recall or deduce some results on deformed Wigner matrices  from \cite{BC, CRMTA}.
In Section \ref{sectionvaleurspropres},   we first establish  
Theorem \ref{vpcasnondiag}, sticking to  the   approach of \cite{CRMTA}.
  The last Section is an appendix reminding the reader about Poicar\'e inequality and concentration phenomenon. \\
  
\noindent

For any integer number $k$, we will say that a matrix-valued    function $ f_N$ on $\mathbb{C}\setminus \R$
is  $O\left(\frac{1}{N^k}\right)$ if there exists a polynomial Q with nonnegative coefficients and an integer number $d$ such that for all large N,  for any $z$ in $\mathbb{C}\setminus \R$, 
$$\Vert f_N(z)\Vert \leq  \frac{Q(|\Im z |^{-1})(\vert z \vert +1)^d}{N^k}.$$
 For a family of  functions $f_N^{(i)}$, $i \in \{1,\ldots,N\}^2$, we will set $f_N^{(i)}=O^{(u)} \left(\frac{1}{N^k}\right)$ if for each $i$, $f_N^{(i)}=O \left(\frac{1}{N^k}\right)$ and moreover one can find a bound of 
the norm of each $f_N^{(i)}$  as above involving a common polynomial $Q$ and a common $d$, that is not depending on $i$. \\
For two sequences $(X_N)_N$ and $(Y_N)_N$ of random variables, $X_N = Y_N + o_{\mathbb P} (1)$ means that $X_N - Y_N \longrightarrow_{N \rightarrow \infty} 0$ in probability.

\section{Preliminaries}\label{prelim}
\subsection{Basic identities and inequalities}
For a matrix $A \in M_N({\mathbb{C}})$ and $I$ and $J$ non empty subsets of $\{1,\ldots ,N\}$, we denote by $A_{I \times J}$ the submatrix of $A$ obtained by keeping the rows with indices $i \in I$ and columns with indices $j \in J$. We set $A_I := A_{I \times I}$. 
\begin{proposition}[Schur inversion formula]\label{Schur} Let $I$ be a non-empty subset of $\{1,\ldots ,N\}$ 
	and $A\in M_N({\mathbb{C}})$ such that $A_{I}$ is invertible, then $A$ is invertible if and only if $A_{I^c}-A_{I^c\times I}A_{I}^{-1}A_{I\times I^c}$ is invertible, 
	in which case the following formulas hold:
	$$(A^{-1})_{I}=(A_I)^{-1}+ (A_I)^{-1}A_{I\times I^c}(A_{I^c}-A_{I^c\times I}A_{I}^{-1}A_{I\times I^c})^{-1}A_{I^c\times I}(A_I)^{-1},$$
	$$(A^{-1})_{I \times I^c}=-(A_I)^{-1}A_{I\times I^c}(A_{I^c}-A_{I^c\times I}A_{I}^{-1}A_{I\times I^c})^{-1},$$
	$$(A^{-1})_{I^c \times I}=-(A_{I^c}-A_{I^c\times I}A_{I}^{-1}A_{I\times I^c})^{-1} A_{I^c\times I}(A_I)^{-1},$$
	$$(A^{-1})_{I^c}=(A_{I^c}-A_{I^c\times I}A_{I}^{-1}A_{I\times I^c})^{-1}.$$
\end{proposition}
\begin{lemma}\label{majcarre}
For any matrix $B \in  M_N(\mathbb{C})$ 
and for any fixed $k$, we have \begin{equation}\label{O}  \sum_{l =1}^N |B_{lk} |^2 \leq  ||B ||^2\end{equation}
(or equivalently \begin{equation}\label{l} \sum_{l =1}^N |B_{kl} |^2 \leq  ||B ||^2.)\end{equation}
Therefore, we have 
 \begin{equation}\label{lp}
 \frac{1}{N} \sum_{k,l =1}^N |B_{kl} |^2 \leq ||B ||^2.
 \end{equation}
\end{lemma}
\begin{proof}
Note that \begin{eqnarray*} \sum_{l =1}^N |B_{lk} |^2 &=& \Tr_N(BE_{kk} B^*)\\&=&\Tr_N(  B^*B E_{kk})\\&\leq& 
 \Vert B \Vert^2 \Tr_N ( E_{kk})= \Vert B \Vert^2  .\end{eqnarray*}
Now, since  $$ \sum_{l =1}^N|B_{kl} |^2=\sum_{l =1}^N |B_{kl}^* |^2 = \sum_{l =1}^N |(B^*)_{lk} |^2  \; \mbox{and} \; \Vert B^* \Vert= \Vert B \Vert,$$
\eqref{O} and \eqref{l} can be deduced from each other thanks to conjugate transposition.
Finally \eqref{l} readily yields \eqref{lp}.
\end{proof}
\begin{lemma}\label{dvptdet}[Lemma A2 \cite{CRMTA}]
Let $A$ and $H$ be $m\times m$ matrices such that, for some $K>0$, \begin{equation}\label{K} \left\|A\right\| \leq K, \; \left\|H\right\| \leq K. \end{equation} Then 
$$ \det (A+H)=  \det (A)+\Tr_m \left( ^t com(A) H \right)+\epsilon,$$
where $com(A)$ denotes the comatrix of $A$, and  there exists a constant  $C_{m,K}>0$, only depending on $m$ and $K$,  such that  $\left| \epsilon\right| \leq C_{m,K} \left\| H\right\|^2.$
\end{lemma}
We will often use the following obvious facts that 
for any $z \in \C \setminus \mathbb{R}$, for any $N \times N $ Hermitian matrix $H$,
\begin{equation}\label{majres} \Vert \left(zI_N-H\right)^{-1}\Vert  \leq |\Im z|^{-1}\end{equation} 
 and for any probability measure $\nu$ on $\R$, the Stieltjes transform  $g_\nu$ satisfies for any $z \in \C \setminus \R$,
$\Im z \Im g_\nu(z) <0$, $\vert g_\nu(z) \vert \leq \vert \Im z \vert^{-1}$, and for any $\alpha\geq 0$, any $x\in \R$,
\begin{equation} \label{majts} \frac{1}{\vert z-\alpha g_\nu(z) -x\vert } \leq \vert \Im z \vert^{-1}.\end{equation}
\subsection{Helffer-Sj\"{o}strand's calculus}\label{HSc}
\subsubsection{Helffer-Sj\"{o}strand's representation  formula}
We recall  Helffer-Sj\"{o}strand's representation  formula (see \cite[Proposition C.1]{BG-K}): let $f \in C^{k+1}(\mathbb R)$ with compact support and $M$ be a Hermitian matrix; we have
\begin{equation} \label{HS}
 f(M) = \frac{1}{\pi} \int_{\mathbb C} \bar{\partial} F_k(f)(z)\ (M-z)^{-1} d^2z
 \end{equation}
where $d^2 z$ denotes the Lebesgue measure on $\mathbb C$,
\begin{equation} \label{defF_k}
F_k(f)(x+iy) = \sum_{l=0}^k \frac{(iy)^l}{l!} f^{(l)}(x) \chi(y)
\end{equation}
where $\chi : \mathbb R \to \mathbb R^+ $ is a smooth compactly supported function such that $\chi \equiv 1$ in a neighborhood of 0, and $\bar{\partial} = \frac{1}{2} (\partial_x +i \partial_y)$. \\
The function $F_k(f)$ coincides with $f$ on the real axis and is an extension to the complex plane. \\
Moreover \begin{equation}  \label{barpartial}
\bar{\partial} F_k(f)(x+iy)=\frac{1}{2}  \frac{(iy)^k}{k!} f^{(k+1)}(x) \chi(y)+ \frac{i}{2} \sum_{l=0}^k \frac{(iy)^l}{l!} f^{(l)}(x) \chi^{'}(y)\end{equation}
Thus, since $\chi \equiv 1$ in a neighborhood of 0,  we have that, in a neighborhood of the  real axis, 
\begin{equation} \label{real-axis}
 \bar{\partial} F_k(f)(x+iy) =\frac{1}{2} \frac{(iy)^k}{k!} f^{(k+1)}(x)  = O(|y|^k) \mbox { as } y \rightarrow 0. \end{equation}
\subsubsection{Computation of Helffer-Sj\"{o}strand's integral}
\begin{proposition} \label{prop-calculHS}
Let $h$ be a smooth function with compact support in  $(\rho_\theta - 2\delta, \rho_\theta + 2 \delta)$ and satisfying $h \equiv 1$ on $[\rho_\theta - \delta, \rho_\theta +  \delta]$.
Let $\chi$  be a compactly supported function on $(-L, L)$, and $\chi = 1$ around $0$. We denote by $D = (\rho_\theta - 2\delta, \rho_\theta + 2 \delta)\times (-L,L)$. \\
Let $\phi$ be a meromorphic function in $D$, with a pole in $\rho_\theta$.
Then,
\begin{equation} \label{calculHS}
I(\phi) := \frac{1}{\pi} \int_{\mathbb C} \bar{\partial} F_k(h)(z)\ \phi(z) d^2z = - {\rm Res}(\phi, \rho_\theta)
\end{equation}
where $F_k(h)$ is defined in \eqref{defF_k}. ${\rm Res} (\phi, \rho_\theta)$ denotes the residue of the function $\phi$ at the point $\rho_\theta$.
\end{proposition}
\begin{proof} Let $\epsilon $ small enough such that $F_k(h)(z) = 1$ for $z \in B(\rho_\theta, \epsilon)$.
Set $D_\epsilon = D \backslash B(\rho_\theta, \epsilon)$. $\phi$ is holomorphic on $D_\epsilon$.
Since $ F_k(h)$ has compact support in $D$, we have,
\begin{eqnarray*}
0 = \int_{\partial D} F_k(h)(z)\ \phi(z) dz &= & \int_{\partial D_\epsilon} F_k(h)(z) \phi(z) dz + 
\int_{\partial B(\rho_\theta, \epsilon)} F_k(h)(z)\ \phi(z) dz  \\
&=& 2i  \int_{D_\epsilon}  \bar{\partial}F_k(h)(z) \  \phi(z) d^2z + \int_{\partial B(\rho_\theta, \epsilon)}  \phi(z) dz 
\end{eqnarray*}
where the first term is obtained by Green's formula using that $\bar{\partial}\phi(z) = 0$ on $D_\epsilon$.
$$   \int_{\ D_\epsilon}  \bar{\partial}F_k(h)(z) \  \phi(z) d^2z\vers_{\epsilon \rightarrow 0}  \pi I(\phi)$$
and 
$$ \int_{\partial B(\rho_\theta, \epsilon)} F_k(h)(z)\ \phi(z) dz = 2i\pi {\rm Res} (\phi, \rho_\theta). \quad $$
\end{proof}
\subsection{Tightness criterion for a sequence of random analytic processes}

We recall here some results from \cite{S}. Let $D \subset \mathbb{C}$ be an open  set in
the complex plane. Denote by ${\cal H} (D)$  the space of complex analytic functions in $D$, endowed with the uniform topology on compact set. For $f \in {\cal H} (D)$ and $K$ a compact set of D, we denote $\Vert f \Vert_K =\sup_{z\in K} \left| f(z)\right|$. 
The space ${\cal H} (D)$ is equipped with the (topological)
Borel $\sigma$-field ${\cal B}({\cal H} (D))$ and the set of probability measures on $({\cal H} (D); {\cal B}({\cal H} (D)))$ is
denoted by ${\cal P}({\cal H}(D))$.
By a random analytic function on D we mean an ${\cal H} (D)$-valued
random variable  on a probability space. The probability
law of a random analytic function is uniquely determined by its finite dimensional
distributions.
\begin{proposition}\label{criterion}
(Proposition 2.5. in \cite{S}) Let $f_n$ be a sequence of random analytic functions in $D$. If $\Vert f_n\Vert_K$ is tight for any compact set K, then ${\cal L}(f_n)$ is
tight in ${\cal P}({\cal H}(D))$.
\end{proposition}
Using that, by Markov's inequality,   for any $C>0$ and any $r>0$, 
\begin{equation}\label{ineg0}\mathbb{P} \left( \left\| f_n \right\|_K >C \right) \leq \frac{1}{C^r} \mathbb{E} \left( \left\| f_n \right\|^r_K \right),\end{equation}
the following  lemma turns out to be useful to prove  tightness results.

\begin{lemma}\label{Shirai}(lemma 2.6 \cite{S})
For any compact set K in D, there exists $\delta > 0$ such that
$$\left\| f \right\|_K^r \leq (\pi \delta^2)^{-1} \int _{\overline{K_\delta}} \left| f(z)\right|^r m(dz), \; f \in {\cal H}(D),$$
for any $r > 0$, where $\overline{K_\delta}\subset D$ is the closure of the $\delta$-neighborhood of K  and $m$ denotes the Lebesgue measure.
\end{lemma}
\subsection{Free convolution with a semicircular distribution}\label{freeconv}

Let $\tau $ be a probability measure on $\R$. 
Its Stieltjes transform $g_\tau: z\mapsto \int_{\R} \frac{1}{z-x}d\tau(x)$ is analytic on the complex upper half-plane $\C^+$. 
There exists a domain $$D_{\alpha , \beta } = \{ u+iv \in \C, |u| < \alpha v, v > \beta \}$$
on which $g_\tau $ is univalent. 
Let $K_\tau $ be its inverse function, defined on $g_\tau (D_{\alpha , \beta })$, and 
$$R_\tau (z) = K_\tau (z) - \frac{1}{z}.$$
Given two probability measures $\tau $ and $\nu $, 
there exists a unique probability measure $\lambda $ such that
$$R_\lambda = R_\tau + R_\nu $$
on a domain where these functions are defined. 
The probability measure $\lambda $ is called 
the free convolution of $\tau $ and $\nu $ and denoted by $\tau \boxplus \nu $. 

The free convolution of probability measures has an important property, 
called subordination, which can be stated as follows: 
let $\tau $ and $\nu $ be two probability measures on $\R$; 
there exists an analytic map $\omega_{\tau,\nu}: \C^+ \rightarrow \C^+$
such that $$\forall z \in \C^+ ,  ~~~~g_{\tau \boxplus \nu}(z)= g_\nu (\omega_{\tau,\nu}(z)).$$
This phenomenon was first observed by D. Voiculescu under a genericity assumption in \cite{voic-fish1}, 
and then proved in generality in \cite{Biane98} Theorem 3.1. 
Later, a new proof of this result was given in \cite{BelBer07}, 
using a fixed point theorem for analytic self-maps of the upper half-plane. \\

In \cite{Biane97b}, P. Biane provides a deep study of the free convolution by a semicircular distribution. 
We first recall here some of his results that will be useful in our approach.
\noindent Let $\nu $ be a probability measure on $\R$ and $\mu_{sc}$ the semicircular distribution defined by \eqref{sc}. For any $z\in \C\setminus{\rm supp} (\mu_{sc}\boxplus \nu)$, the subordination function $\omega_{\mu_{sc}, \nu }$ is given by
\begin{equation}\label{omega} 
\omega_{\mu_{sc}, \nu }(z)=z-\sigma ^2g_{\mu _\sigma \boxplus \nu }(z).
\end{equation}
In the following, we will  denote $\omega_{\mu_{sc}, \nu } $ by $\omega$.
For any $x\in \C\setminus{\rm supp} (\nu)$, define  \begin{equation}\label{H}H(x) =x + \sigma^2 g_\nu(x).\end{equation} We have for any $x\in \C\setminus{\rm supp} (\mu_{sc} \boxplus \nu)$, \begin{equation}\label{composition}H(\omega(x))=x. \end{equation}
More precisely, the following one to one correspondance  holds:
\begin{equation}\label{correspondance}\mathbb{R}\setminus{\rm supp} (\mu_{sc} \boxplus \nu) \begin{array}{cc}\stackrel{{\omega}}{\longrightarrow} \\
 \stackrel{\longleftarrow}{{H}}
 \end{array}  ~\{ u \in \mathbb{R} \setminus {\rm supp~} \nu, \int \frac{1}{(u-x)^2}d\nu(x) <\frac{1}{\sigma^2}\}.\end{equation}
Assume that $A_N$ satisfies {\bf (A)}.
Let 
$\{ \theta _j, \, 1 \leq j \leq q\} $
be the spiked eigenvalue  of $A_{N}$ outside ${\rm supp}(\nu)$.
Furthermore, for all $\theta _j \in \Theta _{\sigma , \nu }$, where $\Theta _{\sigma , \nu }$ is defined by \eqref{cond-outlier}, we set 
\begin{eqnarray}
\rho _{\theta _j}:=H(\theta _j)=\theta _j+\sigma ^2g_\nu (\theta _j)
\end{eqnarray}
which is outside the support of $\mu _{sc} \boxplus \nu $ according to \eqref{correspondance}, 
and we define
$$
K_{\sigma , \nu }(\theta _1, \ldots , \theta _{q}):=
{\rm supp}(\mu _{sc} \boxplus \nu )\bigcup \left\{ \rho _{\theta _j}, \, \theta _j \in \, \Theta _{\sigma , \nu }\right\} .
$$
An important consequence of \eqref{correspondance} is the following
\begin{proposition}\label{Theo-InclSupportN}[Theorem 2.3 \cite{CP}] 
For any $\delta > 0$, 
$${\rm supp}(\mu _{sc} \boxplus \mu _{A_{N}})\subset 
K_{\sigma , \nu }(\theta _1, \ldots , \theta _{q}) + (-\delta, \delta),$$ 
when $N$ is large enough.
\end{proposition}


\subsection{Central limit theorem for processes of   matrix valued  random quadratic forms  }
\begin{lemma}[cf Lemma 2.7 \cite{BS}]\label{lemFQ} There exists $C>0$ such that 
for any  $N \times N$  deterministic matrix $B=(b_{ij})_{1\leq i,j\leq N}$, any random vectors $Y=\begin{pmatrix} y_1\\ \vdots \\ y_N \end{pmatrix}$  in  $\mathbb{C}^N$ with i.i.d.  standardized entries (~$\mathbb{E}(y_i)=0$, $\mathbb{E}( \vert y_i \vert^2)=1$, $\mathbb{E}(  y_i ^2)=0$) such that  $\mathbb{E}( \vert y_i \vert^4)< \kappa$  and any   independent copy $X$  of $Y$, one has
$$\mathbb{E}( \vert Y^*BY- \Tr_N( B) \vert^2 ) \leq C \kappa\Tr_N (B ^*B),$$
$$\mathbb{E}( \vert Y^*BX\vert^2 ) \leq C \kappa \Tr_N (B ^*B).$$
\end{lemma}
In  \cite{NY}, the authors establish the following variation around the central limit theorem for martingales.
\begin{lemma}[Lemma 5.6 \cite{NY}]\label{NaYa}
Suppose that for each $n$ $(Y_{nj}; 1\leq j \leq r_n)$ is a  $\mathbb{C}^d$-valued martingale difference sequence  with respect to the increasing $\sigma$-field $\{\mathcal {G}_{n,j}; 1\leq j\leq r_n\}$ having second moments. Write
$$Y_{nj}^T =( Y_{nj}^1,\ldots, Y_{nj}^d).$$
Assume moreover that 
$(\Theta_n(k,l))_n$ and $(\tilde \Theta_n(k,l))_n$ are uniformly bounded sequences of complex numbers, for $1\leq k,l\leq d.$ If 
\begin{equation}\label{Theta} \sum_{j=1}^{r_n}\mathbb{E}\left( Y_{nj}^k \bar{Y}_{nj}^l \; \left|\;\mathcal{G}_{n,j-1}\right. \right) -\Theta_n(k,l) \overset{P}{\underset{n\rightarrow +\infty}{\longrightarrow}} 0, \end{equation}
\begin{equation}\label{Thetatilde} \sum_{j=1}^{r_n}\mathbb{E}\left( Y_{nj}^k  Y_{nj}^l \; \left| \;\mathcal{G}_{n,j-1}\right. \right) -\tilde \Theta_n(k,l) \overset{P}{\underset{n\rightarrow +\infty}{\longrightarrow}} 0, \end{equation}
and for each $\epsilon >0$,
 $\sum_{j=1}^{r_n} \mathbb{E}\left( \Vert Y_{nj}\Vert^2 \1_{\Vert  Y_{nj}\Vert>\epsilon}\right)\rightarrow_{n\rightarrow +\infty} 0$,
then, for every bounded continuous function $f: \mathbb{C}^d \rightarrow \mathbb{R}$,
\begin{equation}\label{cvloi} \E f(\sum_{j=1}^{r_n}Y_{nj})-\E f(Z_n)  \underset{n\rightarrow +\infty}{\longrightarrow}0,\end{equation}
where $Z_n$ is a $\C^d$-valued centered Gaussian random vector with parameters 
$$\E(Z_nZ_n^*)=(\Theta_n(k,l))_{k,l}\; \mbox{and} \; \E(Z_nZ_n^T)=(\tilde \Theta_n(k,l))_{k,l}.$$
\end{lemma}
Following the lines of the proof of the central limit theorem for  quadratic forms by Baik and Silverstein in the appendix of \cite{CDF} and using  Lemma \ref{NaYa}, we will establish the following extension. 
\begin{proposition} \label{theoTCLquadratique} Let $(z_1, \ldots, z_q)$ be in $I^q$, where $I$ is a subset of $\mathbb{C}$ such that $\forall z \in I, \bar z \in I$.
 For any $z$ in $\{z_1, \ldots, z_q\}$, let $B(z)=(b_{ij}(z))$ be a $N \times N$   matrix such that $(B(z))^*=B(\bar z)$ and 
there exists a constant $a>0$ (not depending on $N$) such that for any $z$ in $\{z_1, \ldots, z_q\}$, $||
  B(z)|| \leq a$. Let $p$ be a fixed integer number and 
$Y_N=(y_{ij})_{1\leq i\leq N, 1 \leq j \leq p}$ be a $N\times p$ matrix  which contains i.i.d. complex
standardized entries with bounded fourth
moment and such that $\mathbb E (y_{11}^2)=0$.
  Set $$V_N= \Big( ({1}/{\sqrt N})(Y_N ^* B(z_1) Y_N -
{\rm{Tr}_N}( B(z_1)) I_p), \ldots,({1}/{\sqrt N}) (Y_N ^* B(z_q) Y_N -
{\rm{Tr}_N} (B(z_q))I_p) \Big ).$$
If there exists  uniformly bounded sequences $(f_N)_N$ and $(g_N)_N$ of functions on $I^2$ such that for any $z_l$, $z_k$, $$\frac{1}{N} \sum_{i=1}^N  
(B(z_l))_{ii} (B(z_k))_{ii}=f_N(z_l,z_k)+o(1),$$
and $$\tr_N(B(z_l)B(z_k))= g_N(z_l,z_k)+o(1),$$
then $d_{LP}\left( V_N , {\mathcal V}_N \right) \rightarrow 0$ where 
$ {\mathcal V}_N=({\cal G}_N(z_1),\ldots, {\cal G}_N(z_q))$,  ${\cal G}_N$ is  a centered Gaussian matrix valued process   whose distribution is given as follows : \\
 1) the processes $(({\mathcal G}_{N}(z))_{ij})_z$ for $1 \leq i \leq j \leq p$ are independent, and $({\mathcal G}_{N}(z))_{ji} = \overline{({\mathcal G}_{N}(\bar{z}))_{ij}}$. \\
 2) For $i \leq p$,
 \begin{eqnarray*}
 \mathbb E(({\mathcal G}_{N}(z_l))_{ii} ({\mathcal G}_{N}(z_k))_{ii} ) &= &(\E(\vert y_{11}\vert^4 -2)\left(f_N(z_l,z_k) \right) \nonumber\\
&&  +
 g_N(z_l,z_k)
\end{eqnarray*}
and
\begin{equation}
\mathbb E(({\mathcal G}_{N}(z_l))_{ii} \overline{({\mathcal G}_{N}(z_k))_{ii})} = \mathbb E(({\mathcal G}_{N}(z_1))_{ii}({\mathcal G}_{N}(\bar{z_2}))_{ii})
\end{equation}
2) For $1 \leq i \not= j \leq p$,
$$ \mathbb E(({\mathcal G}_{N}(z_l))_{ij} ({\mathcal G}_{N}(z_k))_{ij}) = 0$$
and \begin{equation} \label{nondiag}
\mathbb E(({\mathcal G}_{N}(z_l))_{ij}\overline{({\mathcal G}_{N}(z_k))_{ij}} ) = g_N(z_l,\bar{z_k}). \end{equation}
\end{proposition}
\begin{proof}
 First,  for any $B$ in $\{B(z), z=z_1,\ldots,z_q\}$ and any $1\leq s,t \leq p$,  one can write $({1}/{\sqrt N}) \Big ( Y_N ^* B Y_N -
{\rm{Tr}}( B) I_p \Big)_{st}$ as a sum of martingale differences:
\begin{eqnarray*}\lefteqn{(1/\sqrt N)(Y_N^*BY_N-{\rm{Tr}} (B) I_p)_{st}}\\&=&(1/\sqrt N) \sum_{i=1}^N
\bigg((\bar y_{is}y_{it}-\delta_{st})b_{ii}
+\bar y_{is}\sum_{j<i}y_{jt}b_{ij}+\bar y_{is}\sum_{j>i}y_{jt}b_{ij}\bigg)\\&
=&(1/\sqrt N) \sum_{i=1}^N
\bigg((\bar y_{is}y_{it}-\delta_{st})b_{ii}
+\bar y_{is}\sum_{j<i}y_{jt}b_{ij}+y_{it}\sum_{j<i}\bar y_{js}b_{ji}\bigg)
\\&=&\sum_{i=1}^N(Z_i(B))_{st}
\end{eqnarray*}
where
$$(Z_i(B))_{st}=(1/\sqrt N)
\bigg((\bar y_{is}y_{it}-\delta_{st})b_{ii}
+\bar y_{is}\sum_{j<i}y_{jt}b_{ij}+y_{it}\sum_{j<i}\bar y_{js}b_{ji}\bigg).
$$
Let $\mathcal F _{N,i}$  be the
$\sigma$-field generated by $\{y_{1s},\ldots,y_{is}, 1\leq s \leq p\}$. 
Let also $\mathbb{E}_i(\cdot)$ denote conditional expectation with
respect to $\mathcal F _{N,i}$. It is clear that  for any $B$ in $\{B(z), z\in I\}$, $Z_i(B)\in M_{p}(\mathbb{C})\simeq
 \mathbb{C}^{p^2}$ is measurable with respect to $\mathcal F _{N,i}$ and satisfies $\mathbb E_{i-1} ( Z_{i}(B))=0$.\\

We will show that the conditions of Lemma \ref{NaYa} are met for the $\mathbb{C}^{q p^2}$-valued martingale difference sequence 
$Y_{Ni}^T=(Z_i(B(z_1)),\ldots,Z_i(B(z_q))).$

\noindent Write $(Z_i(B))_{st}=X_1^i+X_2^i$, with $X_1^i=(1/\sqrt
N)
(\bar y_{is}y_{it}-\delta_{st})b_{ii}$. Then for $\epsilon>0$,
\begin{equation}\label{terme1}
\sum_{i=1}^N\mathbb{E}(|X^i_1|^2 \, 1_{(|X_1^i|\ge\epsilon)})\leq
a^2\mathbb{E}(|\bar y_{1s}y_{1t} -\delta_{st}|^2 \, 1_{\{|\bar y_{1s}y_{1t} -\delta_{st}|^2\ge\sqrt N\epsilon/a\}})\to0
\end{equation}
as $N \to \infty$, by dominated convergence theorem.\\
\noindent We have
\begin{eqnarray*}
\lefteqn{\mathbb{E}|\sum_{j<i}y_{jt}b_{ij}|^4}\\ & = & \mathbb{E}(|y_{1t}|^4\sum_{j<i}|b_{ij}|^4)
+2 \mathbb{E} (\sum_* |b_{ij_1}|^2|b_{ij_2}|^2) + \mathbb{E}(|y_{1t}^2|^2\sum_*
\vert b_{ij_1} \vert^2 \vert b_{ij_2}\vert^2 )\crcr
 & \leq &
 \mathbb{E}|y_{1t}|^4 \mathbb{E}[(\max_{i,j} \vert b_{ij}\vert)^2 \max_i (BB^*)_{ii}]+(2+\mathbb{E}|y_{1t}^2|^2)
 \mathbb{E}[(BB^*)_{ii}^2] \crcr
 & \leq & a^4 \big [ \mathbb{E}|y_{1t}|^4 + 2+\mathbb{E}|y_{1t}^2|^2 \big ]
\end{eqnarray*}
where the sum $\underset{*}{\sum}$ is over $\{j_1<i, \, j_2<i, \, j_1 \not = j_2 \}$.
Therefore $\mathbb{E}|X_2^i|^4=o(N^{-1})$ so that for any $\epsilon>0$,
\begin{equation}\label{terme2}
\sum_{i=1}^N\mathbb{E}(|X_2^i|^2 \, 1_{(|X_2^i|\ge\epsilon)})\leq(1/\epsilon^2)
\sum_{i=1}^N\mathbb{E}|X_2^i|^4
\to 0\quad \text{ as } N\to\infty.
\end{equation}
Thus, by (\ref{terme1}), (\ref{terme2}) and (A.4) in \cite{CDF}, $\{Z_i(B)\}$ satisfies the Lindeberg condition of Lemma \ref{NaYa}.\\

Now, we shall verify condition \eqref{Thetatilde} of Lemma \ref{NaYa}. We have for any $B$ and $C$ in $\{B(z), z\in I\}$, for any $1\leq s,t,s',t' \leq p,$
\begin{eqnarray}{\label{sum-cond1}}
\lefteqn{\sum_{i=1}^N\mathbb{E}_{i-1}(Z_i (B))_{st} Z_i(C)_{s't'})} \\
& =&   (1/N)\sum_{i=1}^N \bigg{\{} (\mathbb{E}|y_{11}|^4-1)\delta_{st}\delta_{ss'}\delta_{s't'}+ \delta_{st'}\delta_{s't}(1-\delta_{ss'})) b_{ii}c_{ii}
\\&& +\mathbb{E}(|y_{11}|^2\bar
y_{11})\left[ \delta_{st}\delta_{ss'}\sum_{j<i}y_{jt'}c_{ij} b_{ii} + \delta_{s't'}\delta_{ss'}\sum_{j<i}y_{jt}  b_{ij} c_{ii}\right] \nonumber \\
&&+\mathbb{E}(|y_{11}|^2y_{11})\left[\delta_{st} \delta_{tt'} \sum_{j<i}\bar y_{js'} c_{ji} b_{ii}+\delta_{s't'} \delta_{tt'} \sum_{j<i}\bar y_{js}b_{ji} c_{ii}\right]\\ 
&&
+ \delta_{s't}\sum_{j<i}\bar y_{js}b_{ji}\sum_{j<i} y_{jt'} c_{ij} + \delta_{st'}\sum_{j<i}\bar y_{js'}c_{ji}\sum_{j<i} y_{jt} b_{ij}\bigg{\}}.
\nonumber
\end{eqnarray}
Let $B_L$ (resp. $B_U$ denote the strictly lower (resp. upper) triangular part of $B$.
We have
using Cauchy-Schwarz's inequality that 
\begin{eqnarray*}
\mathbb{E}|(1/N)\sum_{i=1}^N c_{ii}\sum_{j<i}y_{jt}b_{ij}|^2 &=&
\mathbb{E}|(1/N)\sum_{j=1}^{N-1} y_{jt}\sum_{i>j}c_{ii}b_{ij}|^2\\
& = & (1/N^2) \mathbb{E}
(\sum_{j=1}^{N-1}\sum_{i>j}c_{ii}b_{ij}\sum_{\underline i>j}
\bar c_{\underline i\,\underline i}\overline b_{\underline ij})\\
&=&(1/N^2) \mathbb{E} ( \sum_{i\underline i}c_{ii}\bar c_{\underline
i\,\underline i} (B_LB_L^*)_{i\underline i})\\ & \leq &
\mathbb{E} \big [ (\max_i\vert c_{ii}\vert)^2(1/N) (\sum_{i\underline
i}|(B_LB_L^*)_{i\underline i}|^2)^{1/2} \big ]\\ & =
&\mathbb{E} \big [ (\max_i\vert c_{ii}\vert )^2(1/N)
{\rm{Tr}}((B_LB_L^*)^2)^{1/2} \big ]\crcr & \leq & \mathbb{E} \big
[ (\max_i\vert c_{ii}\vert) ^2(1/\sqrt N)\|B_L\|^2 \big ].
\end{eqnarray*}
We apply the following bound (due to R. Mathias, see \cite{Mt}):
$\|B_L\|\leq\gamma_N\|B\|$ where $\gamma_N=O(\ln N)$, and the bounds
$||B|| \leq a$, $||C|| \leq a$, to conclude that
$$(1/N) \sum_{i=1}^N c_{ii}\sum_{j<i}y_{jt}b_{ij} \overset{P}{\underset{N \rightarrow +\infty}{\longrightarrow}} 0.$$
Thus,
\begin{eqnarray*}
\lefteqn{ \sum_{i=1}^N\mathbb{E}_{i-1}(Z_i (B))_{st} Z_i(C)_{s't'})} \\
& =&   (1/N)\sum_{i=1}^N \bigg{\{}\left[ (\mathbb{E}|y_{11}|^4-1)\delta_{st}\delta_{ss'}\delta_{s't'}+ \delta_{st'}\delta_{s't}(1-\delta_{ss'})\right] b_{ii}c_{ii}
\\ 
&&
+ \delta_{s't}\sum_{j<i}\bar y_{js}b_{ji}\sum_{j<i} y_{jt'} c_{ij} + \delta_{st'}\sum_{j<i}\bar y_{js'}c_{ji}\sum_{j<i} y_{jt} b_{ij}\bigg{\}}+o_{\mathbb P}(1)\\
&  = & \left[(\mathbb{E}|y_{11}|^4-1)\delta_{st}\delta_{ss'}\delta_{s't'}+ \delta_{st'}\delta_{s't}(1-\delta_{ss'})\right]\frac{1}{N} \sum_{i=1}^N  b_{ii}c_{ii}\\&&+ \delta_{s't}(1/N)(Y_N)_s^*B_UC_L(Y_N)_{t'}  + \delta_{st'}(1/N)(Y_N)_{s'}^*C_UB_L(Y_N)_t +o_{\mathbb P}(1) 
\end{eqnarray*}
Besides, from Lemma \ref{lemFQ}  we have
\begin{eqnarray*}
\mathbb{E} \Big |(1/N)((Y_N)_s^*B_UC_L (Y_N)_{t'}-\delta_{st'}{\rm{Tr}} (B_U C_L) ) \Big
|^2 & \leq &
(1/N^2) \mathbb{E} ({\rm{Tr}}(C_L^*B_U^*B_UC_L ) \\
& \leq & K \mathbb{E}\|B\|^2 \|C\|^2\frac{\ln^4N} N\to_{N\to\infty}0.
\end{eqnarray*}
Hence,
\begin{eqnarray*} 
\lefteqn{\sum_{i=1}^N\mathbb{E}_{i-1} (Z_i (B))_{st}(Z_i(C))_{s't'}}\\&=& \left[(\mathbb{E}|y_{11}|^4-1)\delta_{st}\delta_{ss'}\delta_{s't'}+ \delta_{st'}\delta_{s't}(1-\delta_{ss'})\right]\frac{1}{N} \sum_{i=1}^N  b_{ii}c_{ii} \\&&+\delta_{st'}\delta_{s't}\frac{1}{N} \sum_{j<i}  b_{ij}c_{ji}+\delta_{s't}\delta_{st'}\frac{1}{N} \sum_{j<i}  b_{ji}c_{ij} +o_{\mathbb P}(1)\\
&=& \left[(\mathbb{E}|y_{11}|^4-1)\delta_{st}\delta_{ss'}\delta_{s't'}+ \delta_{st'}\delta_{s't}(1-\delta_{ss'})\right]\frac{1}{N} \sum_{i=1}^N  b_{ii}c_{ii} \\&&+\delta_{st'}\delta_{s't}\frac{1}{N} \mbox{Tr}(BC)-\delta_{s't}\delta_{st'}\frac{1}{N} \sum_{i}  b_{ii}c_{ii} +o_{\mathbb P}(1)
\\
&=& \left[(\mathbb{E}|y_{11}|^4-1)\delta_{st}\delta_{ss'}\delta_{s't'}- \delta_{st'}\delta_{s't}\delta_{ss'})\right]\frac{1}{N} \sum_{i=1}^N  b_{ii}c_{ii} \\&&+\delta_{st'}\delta_{s't} \tr_N(BC) +o_{\mathbb P}(1)
.\end{eqnarray*}
Proposition \ref{theoTCLquadratique} readily follows. \end{proof}
\subsection{Preliminary results on deformed Wigner matrices}\label{rappels}
\subsubsection{Preliminary results from \cite{BC} } 
Note that, in Section 5 in \cite{BC}, the authors consider for any fixed integer numbers  $m, r, t$ and any fixed  $m\times m$ Hermitian matrices $\gamma$, $\alpha_1, \ldots, \alpha_r, \beta_1,\ldots, \beta_t$,  the matrix model in $M_m(\C)\otimes M_N(\C)$, $\gamma \otimes I_N + \sum_{v=1}^r \alpha_v \otimes \frac{W_N^{(v)}}{\sqrt{N}} + \sum_{u=1}^t \beta_u \otimes A_N^{(u)} $ where the $W_N^{(v)}$'s are independent more general Wigner matrices than ours and the $A_N^{(u)}$'s are deterministic matrices such that $\sup_N\Vert A_N^{(u)} \Vert <\infty$. Therefore, the results therein apply to our model by choosing $m=1, r=t=1$ and $\alpha_1=1=\beta_1$.\\

For any $z\in \C\setminus \R$, 
$G(z) =[G_{ij}(z)]_{1\leq i,j\leq N}= (zI_N-M_N)^{-1}$ denotes the resolvant of $M_N$. Note that by \eqref{majres}, 
\begin{equation}\label{B} \Vert G(z) \Vert \leq |\Im z|^{-1}.\end{equation} 
 For any $z\in \C\setminus \R$,  define $g_N(z) =\mathbb{E}(\tr_N G(z))$ and denote by $ \tilde g_N(z)$  the Stieltjes transform of $\mu_{sc}\boxplus \mu_{A_N}$. 
\begin{lemma}\label{estimGij}
For any $(i,j)\in \{1,\ldots,N\}^2$, 
$$\E\left( G_{ij}(z)\right)= \left[\left( (z-\sigma^2 \tilde g_N(z))I_N-A_N\right)^{-1}\right]_{ij}+ O^{(u)}\left(\frac{1}{\sqrt{N}}\right).$$
\end{lemma}
\begin{proof}
 According to Corollary 5.5 in \cite{BC}, for any $(i,j)\in \{1,\ldots,N\}^2$, 
\begin{eqnarray*}\lefteqn{\E\left( G_{ij}(z)\right)}\\&=&\left( Y_N(z)\right)_{ij}\\&& + \frac{(1-\sqrt{-1}) \kappa_3}{2\sqrt{2} N\sqrt{N}}\sum_{s,l=1}^N 
\left( Y_N(z)\right)_{il}\left( Y_N(z)\right)_{ss}\left( Y_N(z)\right)_{ll}\E\left( G_{sj}(z)\right)\\&&+O^{(u)}_{ij}(\frac{1}{N})\end{eqnarray*}
where $$Y_N(z)= \left( (z-g_N(z))I_N -A_N\right)^{-1}$$
 and $\kappa_3$ is the third classical cumulant of $\mu$.
Note that $\vert \Im(z-g_N(z))\vert \geq \vert \Im z \vert$ so that, by \eqref{majres} \begin{equation}\label{majY} \Vert Y_N(z) \Vert \leq \frac{1}{\vert \Im z \vert}. \end{equation}
Now,\begin{eqnarray*}\lefteqn{\left|
\sum_{s,l=1}^N 
\left( Y_N(z)\right)_{il}\left( Y_N(z)\right)_{ss}\left( Y_N(z)\right)_{ll}\E\left( G_{sj}(z)\right) \right|}
\\&\leq & \sum_{s,l=1}^N 
\left|\left( Y_N(z)\right)_{il}\right|\left|\left( Y_N(z)\right)_{ss}\right| \left| \left( Y_N(z)\right)_{ll}\right| \left|\E\left(G_{sj}(z)\right) \right|
\\&\leq & 
\vert \Im z \vert^{-2} N \left(\sum_{s=1}^N  \left|\E\left(G_{sj}(z)\right) \right|^2 \right)^{1/2} \left(\sum_{l=1}^N  \left|\left( Y_N(z)\right)_{il}\right|^2 \right)^{1/2}\\&\leq & N \vert \Im z \vert^{-4}
\end{eqnarray*} where we used Lemma \ref{majcarre}, \eqref{majY} and \eqref{B}.
Therefore $$\E\left( G_{ij}(z)\right)=\left( Y_N(z)\right)_{ij}+O^{(u)}_{ij}(\frac{1}{\sqrt{N}}).$$
Now, according to (5.56) in \cite{BC}, $$\left\| Y_N(z) - \tilde Y_N(z) \right\| =O\left(\frac{1}{\sqrt{N}}\right)$$
where  \begin{equation}\label{deftildeY} \tilde Y_N(z) = \left( (z-\tilde g_N(z))I_N -A_N\right)^{-1}.\end{equation}
Lemma \ref{estimGij} readily follows.
 \end{proof}

\begin{lemma}\label{estimgmoinstildeg}
We have $$\forall z\in \C\setminus \R, \;g_N(z)=\tilde g_N(z) +O(\frac{1}{N}).$$
\end{lemma}
\begin{proof} According to 
Proposition 5.8 in \cite{BC}, we have 
$$g_N(z)-\tilde g_N(z)=\left(1 -\tilde g^{'}_N(z)\right) \tilde L_N(z)+O(\frac{1}{N\sqrt{N}})$$
where 
\begin{eqnarray*}\tilde L_N(z)&=&\frac{\kappa_4}{2N^3} \sum_{i,l=1}^N \left(\tilde  Y_N(z)^2\right)_{ll}\left[\left(\tilde  Y_N(z)\right)_{ii}\right]^2\left(\tilde  Y_N(z)\right)_{ll}\\&&+ \frac{\kappa_3(1+\sqrt{-1})}{2\sqrt{2}N^2\sqrt{N}} \sum_{i,l=1}^N \left(\tilde  Y_N(z)^2\right)_{ll}\left(\tilde  Y_N(z)\right)_{ii}\left(\tilde  Y_N(z)\right)_{li}\\&&
+ \frac{\kappa_3(1-\sqrt{-1})}{2\sqrt{2}N^2\sqrt{N}} \sum_{i,l=1}^N \left(\tilde  Y_N(z)^2\right)_{il}\left(\tilde  Y_N(z)\right)_{ii}\left(\tilde  Y_N(z)\right)_{ll}\\&&+ \frac{\kappa_3(1-\sqrt{-1})}{2\sqrt{2}N^2\sqrt{N}} \sum_{i,l=1}^N \left(\tilde  Y_N(z)^2\right)_{ll}\left(\tilde  Y_N(z)\right)_{il}\left(\tilde  Y_N(z)\right)_{ii},
\end{eqnarray*}
and $\tilde Y_N$ is defined by \eqref{deftildeY}.
(Note that in Proposition 5.8 in \cite{BC}, in full generality the  $ \left(\tilde  Y_N(z)\right)_{il}$'s are $m\times m$ matrices which a priori do not commute and  $\tilde G_N(zI_m)$ is a  $m\times m$ matrix too. But in the present paper, since $m=1$, $ \left(\tilde  Y_N(z)\right)_{il}$'s are scalar and obviously commute and $\tilde G_N(zI_m)=\tilde g_N(z)$.)\\
Note that,  $\vert \Im(z-\tilde g_N(z))\vert \geq \vert \Im z \vert$ so that, by \eqref{majres},  $$\Vert \tilde Y_N(z)\Vert \leq \vert \Im z \vert^{-1}.$$
Lemma \ref{estimgmoinstildeg} readily follows by using Cauchy-Schwarz's inequality and \eqref{lp}.
\end{proof}
Lemma 8.7 in \cite{BC} implies in particular the following variance estimates.
\begin{lemma}\label{var}
 $${ \rm Var}(G_{ij}(z))=O^{(u)}(1/N).$$
\end{lemma}
\begin{lemma}\label{vartrace}
$${ \rm Var}(\tr_N G(z))=O(1/N^2)$$
\end{lemma}
The following result is a corollary of  Theorem 1.1 in \cite{BC}.
\begin{proposition} \label{no} [Theorem 1.1 \cite{BC}] Let $[b; c]$ be a real interval such that there exists $ \delta  > 0$ such that, for any
large $N$, $[b+\delta; c+\delta]$ lies outside the support of $\mu_{sc}\boxplus \mu_{A_N}$.
 Then, almost surely, for all large $N$, there is no eigenvalue of
$M_N$ 
in $[b; c]$.
\end{proposition}
\subsubsection{Quantitative asymptotic freeness}\label{qaf}


Let $\epsilon_0>0$ be fixed  such that $d(\rho_\theta, {\rm supp}(\mu _{sc} \boxplus \nu) \cup \{\rho_{\theta_j}, \theta_j\neq \theta\})>\epsilon_0$ and 
$d(\theta, {\rm supp}(\nu) \cup \{\theta_j, \theta_j\neq \theta\}) >\epsilon_0$. Let  $\Omega_N$ be the event on which there is no eigenvalue 
of $\frac{W_{N-p}}{\sqrt{N}} +A_{N-p}$ in $]\rho_\theta -\epsilon_0;\rho_\theta+\epsilon_0[$, $\lambda_{i_0}(M_N)$ is the unique  eigenvalue of $M_N$ in $]\rho_\theta -\epsilon_0/2;\rho_\theta+\epsilon_0/2[$ and $\left\|\frac{W_{N}}{\sqrt{N}}\right\|\leq 3$.
Propositions \ref{Theo-InclSupportN}, \ref{no} applied to $M_{N-p}$,  Proposition \ref{propCVoutlier} and Bai-Yin's theorem lead that
\begin{equation}\label{OmegaN}\lim_{N\rightarrow +\infty} \1_{\Omega_N}=1, ~\text{a.s.}\end{equation}

Denoting by  $\hat G_{N-p}$  the resolvent of the lower right submatrix of size $N-p$ of $M_N$ and $\rho_N$ being defined by \eqref{rhon},  we have the following
\begin{proposition}\label{propdelta}
$$\sqrt{N}\left\{ \tr_{N-p}  \hat G_{N-p}(\rho_N )\1_{ \Omega_{N}} -\int \frac{d\mu_{sc} \boxplus \mu_{A_{N-p}}(x)}{\left( \rho_N -x\right)} \right\}$$ goes to zero in probability. \end{proposition}
 \begin{proof}
 We stick  to the proof of Proposition 5.5 in \cite{CRMTA}.
Using  Proposition \ref{Theo-InclSupportN}, for $N$ large enough,
\begin{equation}\label{distinfini} d(\rho_N, {\rm supp}(\mu _{sc} \boxplus \mu _{A_{N-p}}))>\epsilon_0/2\end{equation} and on $ \Omega_{N}$,
$
d\left(\{\rho_N,\lambda_{i_0}(M_N)\}, \text{spect}\left( \frac{W_{N-p}}{\sqrt{N}}+ A_{N-p}\right)\right)>\epsilon_0/2,$
so that \begin{equation}\label{majnormehat1}
\left\| \hat G_{N-p}(\rho_N) \right\| \leq \frac{2}{\epsilon_0},\; 
\left\| \hat G_{N-p}(\lambda_{i_0}(M_N)) \right\| \leq \frac{2}{\epsilon_0}.\end{equation}
Moreover,  there exists $K>0$ such that for any $x\in {\rm supp}(\mu _{sc} \boxplus \mu _{A_{N-p}})$, $$|  \rho_N- x| \leq K\;
 \mbox{and on }\;  \Omega_{N},  \;\left\| (\rho_N  I_{N-p}- \frac{W_{N-p}}{\sqrt{N}}- A_{N-p}\right\| \leq K.$$ 
 Let $g: \mathbb{R}\rightarrow \mathbb{R}$ be a ${\cal C}^\infty $ function with  support in $\{\epsilon_0/4\leq \vert x \vert \leq 2K\}$
 and such that  $g\equiv 1$ on $\{\epsilon_0/2\leq \vert x \vert \leq K\}$.
 $f:x\mapsto \frac{g(x)}{x}$ is a ${\cal C}^\infty $ function with compact support.  
 Note that 
 \begin{equation}\label{fetphi}\int \frac{d\mu_{sc} \boxplus \mu_{A_{N-p}}(x)}{\left( \rho_N -x\right)}=\int f\left(\rho_N   - x \right)d\mu_{sc} \boxplus \mu_{A_{N-p}}(x)\end{equation} 
 \begin{equation}\label{zerof} \mbox{and 
on }\; \Omega_{N},\; \hat G_{N-p}(\rho_N )= f\left(\rho_N  I_{N-p} -\frac{W_{N-p}}{\sqrt{N}}- A_{N-p} \right). \end{equation}
According to Lemma \ref{estimgmoinstildeg}, for any $z\in \mathbb{C}\setminus \mathbb{R}$, \begin{equation}\label{etoile}\sqrt{N} \tr_{N-p}\mathbb{E}\left[ \hat G_{N-p} (\rho_N  -z)\right] =
 \sqrt{N}\frac{d\mu_{sc} \boxplus \mu_{A_{N-p}}(x)}{\left( \rho_N -z-x\right)} + o^{(z)}(1),\end{equation}
  where there exist polynomials $Q_1$ and $Q_2$ with non negative coefficients and $(d,k) \in \mathbb{N}^2$ such that \begin{equation}\label{doubleetoile}\Vert o^{(z)}(1) \Vert \leq \frac{Q_1(\vert \Im z \vert^{-1})(\vert z \vert +1)^d}{\sqrt{N}} \leq \frac{1}{\sqrt{N}}\frac{Q_2(\vert \Im z\vert)(\vert z \vert +1)^d}{\vert \Im z \vert^k}.\end{equation}

 Therefore, by Helffer-Sj\"{o}strand functional calculus (see Section \ref{HSc}),
 $$\sqrt{N}\tr_{N-p} \mathbb{E}\left( f\left(\rho_N  I_{N-p} - \frac{W_{N-p}}{\sqrt{N}}- A_{N-p} \right)\right) $$
 $$=\frac{1}{\pi} \int_{\mathbb{C}\setminus \mathbb{R}} \bar \partial F_k(f) (z)\sqrt{N}  \tr_{N-p}\mathbb{E}\left[\hat  G_{N-p} (\rho_N -z)\right]  d^2z$$
and $$\sqrt{N}  \int  f\left((\rho_N  - x \right)d\mu_{sc} \boxplus \mu_{A_{N-p}}(x)=\frac{1}{\pi} \int _{\mathbb{C}\setminus \mathbb{R}} \bar \partial F_k(f) (z)\sqrt{N}   \frac{d\mu_{sc} \boxplus \mu_{A_{N-p}}(x)}{\left( \rho_N -z-x\right)}
 d^2z.$$
Hence, using  \eqref{etoile} and \eqref{fetphi}, we can deduce that  \begin{eqnarray*}
\lefteqn{\sqrt{N} \tr_{N-p} \mathbb{E}\left( f\left(\rho_N I_{N-p} - \frac{W_{N-p}}{\sqrt{N}}- A_{N-p} \right)\right)}\\&=& \sqrt{N} \int \frac{d\mu_{sc} \boxplus \mu_{A_{N-p}}(x)}{\left( \rho_N -x\right)}+  \frac{1}{\pi} \int _{z \in \mathbb{C}\setminus \R} \partial F_k(f) (z)o^{(z)}(1) d^2z. \end{eqnarray*}
 Note that since $f$ and $\chi$ are compactly supported, the last integral is an integral on a bounded set of $\C$ and according to \eqref{doubleetoile} and \eqref{real-axis},
 $$\left\|\frac{1}{\pi} \int_{\mathbb{C}\setminus \mathbb{R}} \partial F_k(f) (z)o^{(z)}(1) d^2z \right\|\leq \frac{C}{\sqrt{N}}.$$
Thus, \begin{equation}\label{un}\sqrt{N}\left\{ \mathbb{E}  \tr_{N-p} \left( f\left(\rho_N  I_{N-p} - \frac{W_{N-p}}{\sqrt{N}}- A_{N-p} \right)\right)-  \int \frac{d\mu_{sc} \boxplus \mu_{A_{N-p}}(x)}{\left( \rho_N -x\right)}\right\}  \rightarrow_{N\rightarrow +\infty}0.\end{equation}
Define
 $k: M_{N-p}^{sa}(\mathbb{C}) \rightarrow \mathbb{C}$ by $$k(X)=  \tr_{N-p}\left[  f\left(\rho_N  I_{N-p} - X- A_{N-p}
  \right)\right].$$
 Applying Poincar\'e inequality, we get that 
 $$\mathbb{E}\left( \left| k(\frac{W_{N-p}}{\sqrt{N}}) -\mathbb{E}(k(\frac{W_{N-p}}{\sqrt{N}}))\right|^2 \right) \leq \frac{C}{N} \mathbb{E}\left( \left\| \text{grad} k\left(\frac{W_{N-p}}{\sqrt{N}}\right)\right\|_e^2 \right),$$
with
 $$ \left\| \text{grad} k(X)\right\|_e^2 =\sup_{w\in S_1(M_{N-p}^{sa}(\mathbb{C}))} \left| \frac{d}{dt} k(X+tw)_{\vert_{t=0}} \right|^2.$$
%
Since $f$ is a Lipschitz function on $\mathbb{R}$ with  Lipschitz constant $C_L$, its extension on Hermitian matrices is $C_L$-Lipschitz with respect to the norm $\Vert M\Vert_e=(\Tr_{N-p} M^2)^{1/2}$.
Therefore, 
$$\sup_{w\in S_1(M_{N-p}^{sa}(\mathbb{C}))} \left| \frac{d}{dt} k(X+tw)_{\vert_{t=0}} \right|^2 \leq \frac{C}{N},$$  and then 
$$ \mathbb{E}\left( \left| \sqrt{N}\left\{ k\left(\frac{W_{N-p}}{\sqrt{N}}\right) -\mathbb{E}\left(k\left(\frac{W_{N-p}}{\sqrt{N}}\right)\right)\right\}\right|^2 \right) \leq \frac{C}{N}.$$
It readily follows that  \begin{equation}\label{deux} \sqrt{N} \tr_{N-p} \left\{ f(\rho_N  I_{N-p} - \frac{W_{N-p}}{\sqrt{N}}- A_{N-p}) -   \mathbb{E}  \left( f(\rho_N  I_{N-p} -\frac{W_{N-p}}{\sqrt{N}}- A_{N-p})\right)\right\}=o_{\mathbb{P}}(1).\end{equation}
Proposition \ref{propdelta} follows from \eqref{zerof}, \eqref{un}, \eqref{deux} and \eqref{OmegaN}.
 \end{proof} 
\section{Proof of Theorem \ref{vpcasnondiag}}\label{sectionvaleurspropres}
The approach to prove \eqref{cvprincipal} is the one of \cite{CRMTA}. 
On $\Omega_N$, defined at the beginning of Section \ref{qaf},  we have by Proposition \ref{Schur}, 
 \begin{equation}\label{det}\det  \left( X_p(N)\right) =0,\end{equation}
 where $$X_p(N)=  \lambda_{i_0}(M_N)I_p   -  \frac{W_{p}}{\sqrt{N}} -A_p - \frac{1}{N} Y^*  \hat G_{N-p}( \lambda_{i_0}(M_N))  Y,$$ 
and $\hat G_{N-p}$ is the resolvent of $ \frac{W_{N-p}}{\sqrt{N}}+ A_{N-p}$.
Let $\rho_N$ be as defined by \eqref{rhon}. Using 
the identity 
$$\hat G_{N-p}(\rho_N )-\hat G_{N-p}( \lambda_{i_0}(M_N) )=
(\lambda_{i_0}(M_N)-\rho_N)\hat G_{N-p}(\rho_N ) \hat G_{N-p}( \lambda_{i_0}(M_N) ),$$
we have 
$$X_p(N)
=H_p(N) + X_p^{(0)},$$
where $$X_p^{(0)}=\theta I_p -A_p,$$
\begin{eqnarray*}H_p(N)&=& ( \lambda_{i_0}(M_N)- \rho_N)I_p -\Delta_1(N)-\Delta_2(N) \\ &&+( \lambda_{i_0}(M_N)- \rho_N)r_1(N) - \frac{W_{p}}{\sqrt{N}} -(\lambda_{i_0}(M_N)- \rho_N)^2r_2(N)\end{eqnarray*}
with $$r_1(N)= \frac{1}{N} Y^*\hat G_{N-p}(\rho_N)^2\1_{ \Omega_{N}} Y,$$
\begin{eqnarray*}r_2(N) &=&\frac{1}{N} Y^*\hat G_{N-p}(\rho_N )^2   \hat G_{N-p}(\lambda_{i_0}(M_N))\1_{\Omega_{N}}Y,\end{eqnarray*}
$$\Delta_{1}(N)=\frac{1}{N} Y^*\hat G_{N-p}(\rho_N )\1_{ \Omega_{N}} Y -  \frac{\sigma^2}{N}\Tr_{N-p} \left( \hat G_{N-p}(\rho_N )\1_{ \Omega_{N}} \right),$$
$$\Delta_{2}(N)= \frac{1}{N}\Tr_{N-p}\left( \hat G_{N-p}(\rho_N )\1_{ \Omega_{N}} \right)  - \int \frac{d\mu_{sc} \boxplus \mu_{A_{N-p}}(x)}{\left( \rho_N -x\right)}.$$
\noindent First, by Lemma \ref{lemFQ} we have that,  
$$r_1(N) - \sigma^2\frac{1}{N}\Tr_{N-p}( \hat G_{N-p}(\rho_N )^2)\1_{\Omega_{N}} I_p= o_\mathbb{P}(1).$$
By \eqref{OmegaN}, \eqref{zerof}, \eqref{convrho}, \eqref{distinfini} and asymptotic freeness of $\frac{W_{N-p}}{\sqrt{N}}$ and $A_{N-p}$ (see Theorem 5.4.5 \cite{AGZ}),
\begin{equation}\label{cvquad1} \frac{1}{N}\Tr_{N-p}( \hat G_{N-p}(\rho_N )^2)\1_{ \Omega_{N}} \vers_{N \rightarrow \infty}\int \frac{d\lambda(x)}{\left( \rho_\theta -x\right)^{2}} \mbox{~almost surely}.\end{equation}
Therefore, 
\begin{equation}\label{convr1}
r_1(N) \vers^{\mathbb P}_{N \rightarrow \infty} \sigma^2 \int \frac{d\lambda(x)}{( \rho_\theta -x)^{2}}I_p.
\end{equation}
Now on $\Omega_N$, using \eqref{majnormehat1},
\begin{equation}\label{r2}\Vert r_2(N)\Vert  \leq   \left\| \hat G_{N-p}(\rho_N )\1_{ \Omega_{N}}\right\|^2 \left\| \hat G_{N-p}(\lambda_{i_0}(M_N) )\1_{ \Omega_{N}}\right\| \frac{\Vert Y \Vert^2}{N}\leq 9 \left( \frac{2}{\epsilon_0}\right)^2.\end{equation}
\noindent By Lemma \ref{lemFQ}, \begin{equation}\label{delta1zero}\Delta_1(N) =o_{\mathbb{P}}(1).\end{equation}
\begin{lemma}\label{cvquad}
$$\frac{1}{N-p}\sum_{j=1}^{N-p} \left[ \hat G_{N-p}(\rho_N)_{jj}\right]^2 \1_{\Omega_N}= \frac{1}{N-p}\sum_{j=1}^{N-p}  \left[ [(\theta-A_{N-p})^{-1} ]_{jj}\right]^2 +o_{\mathbb{P}}(1).$$
\end{lemma}
\begin{proof}
By \eqref{I_N} which will be proved below, 
for any $r>0$, $$\frac{1}{N-p}\sum_{j=1}^{N-p} \left[ \hat G_{N-p}(\rho_\theta+ \frac{i}{r})_{jj}\right]^2 \1_{\Omega_N}= \frac{1}{N-p}\sum_{j=1}^{N-p}  \left[ [(\omega(\rho_\theta+ \frac{i}{r})-A_{N-p})^{-1} ]_{jj}\right]^2 +o_{\mathbb{P}}(1),
$$ with $\omega$ defined by \eqref{omega}.
Now, using resolvent identity, one can easily obtain that there exists some constant $C(\epsilon_0)$ such that for any $r>0$ and any $N$, 
$$\left|\frac{1}{N-p}\sum_{j=1}^{N-p} \left[ \hat G_{N-p}(\rho_N)_{jj}\right]^2 \1_{\Omega_N}-\frac{1}{N-p}\sum_{j=1}^{N-p} \left[ \hat G_{N-p}(\rho_\theta+ \frac{i}{r})_{jj}\right]^2 \1_{\Omega_N}\right|$$ $$\leq  C(\epsilon_0)\left(\frac{1}{r}  +\rho_N-\rho\right). $$
Moreover, $\omega(\rho_\theta)=\theta$ so that for all large $N$, $d(\omega(\rho_\theta),  {\rm supp}(\mu_{A_{N-p}})> \epsilon_0/2.$
Now, choose $r_0$ large enough such that 
 for all large $N$, 
$\forall r \geq r_0$, $d(\omega(\rho_\theta+\frac{i}{r}),  {\rm supp}(\mu_{A_{N-p}})> \epsilon_0/4.$
Using resolvent identity, one can easily obtain that there exists some constant $C(\epsilon_0)$ such that 
 for all large $N$, 
$\forall r \geq r_0$, 
$$\left|\frac{1}{N-p}\sum_{j=1}^{N-p}  \left[ [(\omega(\rho_\theta+ \frac{i}{r})-A_{N-p})^{-1} ]_{jj}\right]^2 -\frac{1}{N-p}\sum_{j=1}^{N-p}  \left[ [(\omega(\rho_\theta)-A_{N-p})^{-1} ]_{jj}\right]^2\right|$$$$\leq C(\epsilon_0)/r.$$
Lemma \ref{cvquad} follows by letting $N$ go to infinity and then r go to infinity.\end{proof}
\noindent \eqref{cvquad1}, Lemma \ref{cvquad} and  Proposition \ref{theoTCLquadratique} yield that \begin{equation} \label{Delta1}\sqrt{N}  \Delta_1(N)^2=o_{\mathbb{P}}(1),\end{equation}
Now, one can prove that by Proposition \ref{propdelta}, we have
\begin{equation}\label{Delta2}
\sqrt{N} \Delta_2(N)=o_{\mathbb{P}}(1).
\end{equation}
 \noindent Thus \eqref{convlambda}, \eqref{convrho},  \eqref{convr1}, \eqref{r2}, \eqref{delta1zero} and  \eqref{Delta2} yield that  \begin{equation} \label{H} H_p(N)=o_{\mathbb{P}}(1).\end{equation}
\noindent Therefore, according to  Lemma \ref{dvptdet} (using  \eqref{H}) and  \eqref{det}, with a probability going to one as $N$ goes to infinity,   \begin{eqnarray*} 0& = &\det X_p(N)\\
&=& \det(X_p^{(0)}+H_p(N))\\&=& \det (X_p^{(0)}) +\Tr_p \left[B_{X_p^{(0)}} H_p(N)\right]+\epsilon_N\\&=& \Tr_p \left[B_{X_p^{(0)}} H_p(N)\right]+\epsilon_N,
\end{eqnarray*}
where $$B_{X_p^{(0)}}=^t com(X_p^{(0)}),$$
$$\epsilon_N=O(\Vert H_p(N)\Vert^2).$$  Thus, using 
\eqref{convlambda}, \eqref{convrho}, \eqref{convr1}, \eqref{r2}, \eqref{Delta2} and  \eqref{Delta1},  $$\sqrt{N}\epsilon_N =o_{\mathbb{P}}( \sqrt{N} (\lambda_{i_0}(M_N)-\rho_N)) +o_{\mathbb{P}}(1).$$
Hence, with a probability going to one as $N$ goes to infinity,
$$\sqrt{N} ( \lambda_{i_0}(M_N)-\rho_N) \left[ Tr_p  B_{X_p^{(0)}}(I_p+r_1(N)) +o_{\mathbb{P}}(1)\right]$$ $$
= Tr_p \left[ B_{X_p^{(0)}} \left(\sqrt{N} \Delta_1(N) +W_{p} \right)\right] +o_{\mathbb{P}}(1).$$
 \eqref{cvprincipal} readily follows from  \noindent \eqref{cvquad1}, Lemma \ref{cvquad}, Proposition \ref{theoTCLquadratique}, the independence of $ \Delta_1(N)$ and $W_{p}$.
When $A_{N-p}$ is diagonal, the result follows using \eqref{I_N}.

\section{Proofs of Theorems  \ref{casnondiag} and \ref{propcasnondiag2}}\label{common}
 Let $M_N$ be defined by \eqref{defMN}  with assumptions {\bf(W)} and {\bf (A')}.
We denote by $\lambda_i(A_N)$, resp. $\lambda_i(M_N)$, the eigenvalues of $A_N$, resp. $M_N$ and $u_i$, resp. $v_i$ the normalized associated eigenvectors. 
According to assumption {\bf (A)}, there exists $\delta>0$ such that  for all large $N$,  the distance from $\theta$ to the rest of the spectrum
(that is the other eigenvalues of $A_N$  except $\theta$)  is greater than $\delta$.
Moreover, from Proposition \ref{propCVoutlier}, we know that a.s. \begin{equation}
\label{CVlambdaps}
 \lambda_{i_0}(M_N) \longrightarrow_{N\rightarrow +\infty} \rho_\theta= \theta + \sigma^2 g_\nu(\theta).
\end{equation}  and there exists $\delta_0>0$ such that almost surely for all large $N$,  the distance from $\rho_\theta$ to the rest of the spectrum
(that is the other eigenvalues of $M_N$  except $\lambda_{i_0}(M_N)$)  is greater than $\delta_0$.\\
Throughout this section,  $h$ is a smooth function   with support in $]\rho_\theta-\delta_0/2; \rho_\theta+\delta_0/2[$ which is equal to 1 near $\rho_\theta$.
\subsection{Representation in terms of resolvent}
The aim of this section is to be brought back to the study of the fluctuations of the $p\times p$-matrix valued process
$ \{ G_p(z), z \in \mathbb C \backslash \mathbb R \}$
where $G_p(z)$ denotes the principal submatrix of size $p$ of the resolvant matrix $G(z) = (zI_N-M_N)^{-1}$.
\begin{proposition}\label{backtoresolvent} Almost surely, for all large $N$,

~~
\\
$ \sqrt{N} (|\langle u_{i_0}, v_{i_0}\rangle|^2 -\tau_N(\theta)) $
$$=  -  \frac{1}{\pi} \int_{\mathbb C} \bar{\partial} F_k(h)(z)   \left( P^* \sqrt{N}\left(G_{p}(z)-\Lambda_p(z) \right) P \right)_{11}d^2z , $$
where $ \tau_N(\theta)$ is defined by \eqref{tauN}, \begin{equation}\label{deflambdap}\Lambda_p(z) = (zI_p - A_p - \sigma^2 \tilde{g}_{N-p}(z)I_p)^{-1}\end{equation}
and  $\tilde{g}_{N-p}(z)$ is the Stieltjes transform of $\mu_{sc} \boxplus \mu_{A_{N-p}}$. 
\end{proposition}
Proposition \ref{backtoresolvent} readily follows from the  two  preliminaries Lemmas  \ref{ecritureG} and \ref{ecrituretautilde}.
\begin{lemma}\label{ecritureG} Almost surely, for all large $N$, for any integer number $k$, 
 \begin{equation} \label{HS2} |\langle u_{i_0}, v_{i_0}\rangle|^2 = -  \frac{1}{\pi}  \int_{\mathbb C} \bar{\partial} F_k(h)(z) \left( P^*  G_{p}(z)P \right)_{11} d^2z, \end{equation}
 where $ F_k(h)$ is defined by \eqref{defF_k}.
\end{lemma}
\begin{proof}
Let   $f$ be any smooth function with support in $]\theta-\delta/2; \theta+\delta/2[$ which  is equal to 1 near  $\theta$.
From the formula 
\begin{equation} \label{Trace}
 \Tr_N(h(M_N) f(A_N)) = \sum_{i,j=1}^N h(\lambda_i(M_N)) f(\lambda_j(A_N)) |\langle u_j, v_i \rangle|^2, \end{equation}
we easily deduce that, almost surely, for all large $N$,
\begin{equation} \label{def-h}
 |\langle u_{i_0}, v_{i_0}\rangle|^2 = \sum_{i,j=1}^p(P^*)_{1i} h(M_N)_{ij}P_{j1}.
 \end{equation}
By Helffer-Sj\"{o}strand's representation  formula \eqref{HS}, we can  write $h(M_N)_{ij}$ as
$$
h(M_N)_{ij} = -  \frac{1}{\pi}  \int_{\mathbb C} \bar{\partial} F_k(h)(z) \ G_{ij}(z) d^2z,
$$
so that \begin{equation} \label{HS2}\sum_{i,j=p}^p(P^*)_{1i} h(M_N)_{ij}P_{j1}= -  \frac{1}{\pi}  \int_{\mathbb C} \bar{\partial} F_k(h)(z) \left( P^*  G_{p}(z)P \right)_{11} d^2z. \end{equation}
Lemma \ref{ecritureG} follows from \eqref{def-h} and \eqref{HS2}.
\end{proof}

\begin{lemma}\label{ecrituretautilde} For $N$ large enough, 
\begin{eqnarray*} \tau_N(\theta) &=& 1 - \sigma^2 \int \frac{1}{(\theta -x)^2} d\mu_{A_{N-p}}(x)\\&=&- \frac{1}{\pi} \int_{\mathbb C} \bar{\partial} F_k(h)(z)  \left(  P^*(zI_p - A_p - \sigma^2 \tilde{g}_{N-p}(z)I_p)^{-1}P\right)_{11}d^2z\end{eqnarray*} where $\tilde{g}_{N-p}(z)$ is the Stieltjes transform of $\mu_{sc} \boxplus \mu_{A_{N-p}}$. 
\end{lemma}
\begin{proof}
Note that  $$\left(  P^*(zI_p - A_p - \sigma^2 \tilde{g}_{N-p}(z)I_p)^{-1}P\right)_{11}=\frac{1}{z-  \sigma^2 \tilde{g}_{N-p}(z) -\theta}.$$
Let us  define 
for any $z\in \C\setminus{\rm supp} (\mu_{A_{N-p}})$,   \begin{equation}\label{HNmoinsun}H_{N-p}(z) =z + \sigma^2 g_{\mu_{A_{N-p}}}(z)\end{equation} and  for any $z\in \C\setminus{\rm supp} (\mu_{sc} \boxplus \mu_{A_{N-p}})$,
 \begin{equation}\label{omegaNmoinsun}\omega_{N-p}(z) =z - \sigma^2 g_{\mu_{sc} \boxplus \mu_{A_{N-p}}}(z)=z-\sigma^2 \tilde{g}_{N-p}(z).\end{equation}
Note that, for any $z\in \C \setminus \R$, $|\Im \omega_{N-p}(z)|\geq |\Im z |>0.$
Moreover, according to \eqref{correspondance}, the following one to one correspondance  holds:
$$\mathbb{R}\setminus{\rm supp} (\mu_{sc} \boxplus \mu_{A_{N-p}}) \begin{array}{cc}\stackrel{{\omega_{N-p}}}{\longrightarrow} \\
 \stackrel{\longleftarrow}{{H_{N-p}}}
 \end{array}  ~\{ u \in \mathbb{R} \setminus {\rm supp~} (\mu_{A_{N-p}}), \int \frac{1}{(u-x)^2}d\mu_{A_{N-p}}(x) <\frac{1}{\sigma^2}\}$$
and  for any $x\in \C\setminus{\rm supp} (\mu_{sc} \boxplus \mu_{A_{N-p}})$, $H_{N-p}(\omega_{N-p}(x))=x$.
Hence $\rho_N=H_{N-p}(\theta)=\theta+ \sigma^2 g_{\mu_{A_{N-p}}}(\theta)$ is the single pole of $\frac{1}{z-  \sigma^2 \tilde{g}_{N-p}(z) -\theta}$ in $\C$.
Therefore, \eqref{convrho} and 
 Proposition \ref{prop-calculHS} (used with $\phi(z)= \frac{1}{\omega_{N-p}(z)-\theta)}$) imply that 
\begin{equation} \label{tauNbis}
 - \frac{1}{\pi}  \int_{\mathbb C} \bar{\partial} F_k(h)(z) \ \frac{1}{z- \sigma^2 \tilde{g}_{N-p}(z)-\theta} d^2z=\frac{1}{\omega^{'}_{N-p}(\rho_N)}=H_{N-p}^{'}(\theta)= \tau_N(\theta). \end{equation}
\end{proof}
We now consider the process
\begin{equation}\label{defxiN}  \xi_N(z) = \left( P^* \sqrt{N}\left(G_{p}(z)-\Lambda_p(z) \right) P \right)_{11}.\end{equation} where $\Lambda_p$ is defined by \eqref{deflambdap}.

\subsection{Tightness of the sequence of  processes $\{\xi_N\}_N$} \label{sectiontight}
\begin{proposition}\label{tensiondexiN}
$\xi_{N}$ is tight on  ${\cal H}( \mathbb C\setminus \R)$.
\end{proposition}
\begin{proof}

$\xi_N: z\mapsto  \left( P^* \sqrt{N}\left(G_{p}(z)-\Lambda_p(z) \right) P \right)_{11}$
 is analytic on $\mathbb{C}\setminus \mathbb{R}$. Let $K$ be a compact set  in $\mathbb{C}\setminus \R$.
According to Lemma \ref{Shirai},  there exists $\delta > 0$ such that $\overline{K_\delta}\subset \mathbb{C}\setminus \R$ and for any $r>0$,
$$\left\| \xi_N \right\|_K^r \leq (\pi \delta^2)^{-1} \int _{\overline{K_\delta}} \left| \xi_N(z)\right|^r m(dz).$$
Therefore 
\begin{eqnarray}
\mathbb{E}\left( \left\| \xi_N \right\|_K^r \right)&\leq& (\pi \delta^2)^{-1} \int _{\overline{K_\delta}} \mathbb{E} \left( \left| \xi_N(z)\right|^r \right) m(dz)\\ &\leq&  (\pi \delta^2)^{-1} \sup_{z\in \overline{K_\delta}}  \mathbb{E} \left( \left| \xi_N(z)\right|^r \right) m({\overline{K_\delta}}). \label{ineg1}
\end{eqnarray}

In order to prove the tightness of $\xi_N$, using  \eqref{ineg0} and \eqref{ineg1}, we  are going to show that, for any compact set $K\subset \mathbb{C}\setminus \R$, there exists a constant $C'>0$ such that for all large N,  \begin{equation}\label{maj}\sup_{z\in K} \mathbb{E}\left( \left| \xi_N(z)\right|^2\right)<C'.\end{equation}
We have for any $z_1$ and $z_2$ in $\mathbb{C}\setminus \R$,  
\begin{eqnarray*} 
\lefteqn{\mathbb{E}\left( \xi_N(z_1)\xi_N(z_2)\right) } \\
&=&N \sum_{i,j,u,v=1}^p P^*_{1i}P_{j1} P^*_{1u}P_{v1} \mathbb{E} \left( \left( G_{ij}(z_1)-(\Lambda_p(z_1))_{ij} \right)
\left( G_{uv}(z_2)-(\Lambda_p(z_2))_{uv}\right)\right),\\
\end{eqnarray*} and \begin{eqnarray*}
\lefteqn{\mathbb{E} \left( \left( G_{ij}(z_1)-(\Lambda_p(z_1))_{ij} \right)
\left( G_{uv}(z_2)-(\Lambda_p(z_2))_{uv}\right)\right)}\\&=& \mathbb{E} \left( \left( G_{ij}(z_1)- \mathbb{E} \left( G_{ij}(z_1)\right) \right)
 \left( G_{uv}(z_2)- \mathbb{E} \left( G_{uv}(z_2)\right) \right)\right)\\&&
+  \left[  \mathbb{E} \left( G_{ij}(z_1)\right) -(\Lambda_p(z_1))_{ij}\right]  \left[  \mathbb{E} \left( G_{uv}(z_2)\right) -(\Lambda_p(z_2))_{uv}\right]. 
\end{eqnarray*}
According to Lemma \ref{estimGij} and the block diagonal structure of $A_N$, we have for any $l=1,2$,
\begin{equation}\label{GLambda}\mathbb{E} \left( G_{p}(z_l)\right) =\left((z_l-\sigma^2\tilde g_{N}(z_l))I_p -A_p\right)^{-1} +O(1/\sqrt{N}),\end{equation}
where $\tilde g_{N}$ is the Stieltjes transform of $\mu_{A_N}\boxplus \mu_{sc}$. 
First set 
$$\check G_{N-p} (w)=\left( z I_{N-p}  -\frac{ W_{N-p}}{\sqrt{N-p}}-  A_{N-p} \right)^{-1}.$$
By \eqref{majres}, we have \begin{equation}\label{chapeau}
\Vert \check G_{N-p}(z) \Vert \leq \vert \Im z \vert^{-1}.\end{equation}
Note that, 
\begin{eqnarray}\hat G_{N-p}(z)&=&\check G_{N-p}(z)\nonumber \\ &&+\frac{p}{\sqrt{N-p}(\sqrt{N} +\sqrt{N-p})}\nonumber\\&& ~~~~~~~\times\left( I_{N-p} - \hat G_{N-p}(z)  \left(z I_{N-p} - A_{N-p} \right)\right)
\check G_{N-p}(z). \label{comparaisonhat}\end{eqnarray}
Now, Lemma \ref{estimgmoinstildeg} yields that
$$\tilde g_{N}(z)= \mathbb{E} \left( \tr_N G(z)\right)+ O(1/N)$$ and $$\tilde g_{N-p}(z)= \mathbb{E} \left( \tr_{N-p}  \check G_{N-p}(z)\right)+ O(1/N),$$
and then, using \eqref{comparaisonhat}, that \begin{equation} \label{secondtermtr} \tilde g_{N-p}(z)= \mathbb{E} \left( \tr_{N-p} \hat  G_{N-p}(z)\right)+ O(1/N).\end{equation}   Since by (A.1.12) in \cite{BS10}, we have $\mathbb{E} \left( \tr_N G(z)\right)=\mathbb{E} \left( \tr_{N-p} \hat G_{N-p}(z)\right)+ O(1/N)$, we can deduce that $\tilde g_N(z) = \tilde g_{N-p}(z)+O(1/N),$ and thus that, for any $1\leq i,j\leq p$, $$\mathbb{E} \left( G_{ij}(z_l)\right) =(\Lambda_p(z_l))_{ij} +O(1/\sqrt{N}).$$
 Moreover,  by Lemma \ref{var}, $${ \rm Var}(G_{ij}(z_l))=O(1/N).$$
It readily follows that there exist polynomials $P_1$ and $P_2$ with nonnegative coefficients such that 
$$\mathbb{E}\left(  \xi_N(z_1)\xi_N(z_2) \right) \leq P_1\left( \left| \Im z_1 \right|^{-1}\right)P_2\left( \left| \Im z_2 \right|^{-1}\right)$$
and then there exists some polynomial $P_3$ with nonnegative coefficients such that for all large N and all $z\in \mathbb{C}\setminus \R$,
\begin{equation} \label{xi}
\mathbb{E}\left( \left| \xi_N(z)\right|^2 \right) \leq P_3\left( \left| \Im z \right|^{-1}\right).\end{equation}
This implies \eqref{maj}. Therefore, for any compact subset $K$ in $\C \setminus \R$,  
there exists a constant $C>0$ such that for all large N, $\mathbb{E}\left( \left\| \xi_N(z)\right\|_K^2\right)<C$
and the tightness of $\xi_{N}$ in  $\mathcal{H}(\C \setminus \R)$  follows from 
\eqref{ineg0} and Proposition \ref{criterion}.

\end{proof}
\subsection{Finite dimensional distributions of $\xi_N$}
Set $$\nabla_N (z) =\sqrt{N}\left(G_{p}(z)-\Lambda_p(z) \right).$$
We use Proposition \ref{Schur} for the inversion of 
$$zI_N -M_N = \begin{pmatrix} zI_p -\frac{1}{\sqrt{N}} W_p - A_p& - \frac{1}{\sqrt{N}}Y^* \\  -\frac{1}{\sqrt{N}}Y & zI_{N-p} -  M_{N-p} \end{pmatrix} $$
where $M_{N-p}$ is the lower right submatrix of size $N-p$ of $M_N$, 
 leading to:
\begin{equation} \label{Gp}
G_p(z) = (zI_p - \frac{1}{\sqrt{N}} W_p - A_p -  \frac{1}{N}Y^* \hat G_{N-p}(z) Y)^{-1} \end{equation}
where $\hat G_{N-p}$ is the resolvent of $M_{N-p}$.
Thus, 
\begin{equation} \label{schurxi} \nabla_N(z) = G_p(z) (W_p + \sqrt{N}( \frac{1}{N}Y^* \hat G_{N-p}(z) Y - \sigma^2 \tilde{g}_{N-p}(z)I_p)) \Lambda_p(z).\end{equation}

\begin{lemma}\label{comparaison}
Define for   $z \in \mathbb C \backslash \mathbb R$, 
  \begin{equation} \label{tildexiN}\tilde \nabla_N(z)= (zI_p - A_p - \sigma^2 g(z)I_p)^{-1}(W_p +Q_N(z))(zI_p - A_p - \sigma^2 g(z)I_p)^{-1},\end{equation}
where $Q_N(z)$ is   the following matrix of size $p$ : \begin{equation} \label{QNmatrix}
 Q_N(z) =  \frac{1}{\sqrt{N}}( Y^* \hat G_{N-p}(z) Y - \sigma^2 \Tr_{N-p}(\hat G_{N-p}(z)) I_p), z\in  \C \setminus \R
 \end{equation}
For any $z \in \C \setminus \R$, 
$$\nabla_N(z)-\tilde \nabla_N(z)  \vers^{\mathbb P}_{N \rightarrow \infty} 0. $$
\end{lemma}
\begin{proof}

Obviously, we have 
\begin{equation}\label{defLambda}  \Lambda_p(z) \vers_{N \rightarrow \infty} (zI_p - A_p - \sigma^2 g(z)I_p)^{-1},\end{equation}
and
\eqref{GLambda} and  Lemma \ref{var} yield  that
\begin{equation}\label{cvGp}
 G_p(z)  \vers^{\mathbb P}_{N \rightarrow \infty} (zI_p - A_p - \sigma^2 g(z)I_p)^{-1}.
 \end{equation} 
We write\\

\noindent $\sqrt{N}( \frac{1}{N}Y^* \hat G_{N-p}(z) Y - \sigma^2 \tilde{g}_{N-p}(z)I_p)$ \begin{eqnarray}&=&
 \frac{1}{\sqrt{N}}( Y^* \hat  G_{N-p}(z) Y - \sigma^2 \Tr_{N-p}(\hat G_{N-p}(z)) I_p)\nonumber\\&& + \sqrt{N}  \sigma^2 (\tr_{N-p}(\hat G_{N-p}(z)) I_p - \tilde{g}_{N-p}(z)I_p).\label{decomposition}\end{eqnarray}
\eqref{secondtermtr} and  Lemma \ref{vartrace} yield  that \begin{equation}\label{secondterm} \sqrt{N}  \sigma^2 (\tr_{N-p}(\hat G_{N-p}(z)) I_p - \tilde{g}_{N-p}(z)I_p)  \vers^{\mathbb P}_{N \rightarrow \infty}0. \end{equation}
Lemma \ref{comparaison} readily follows from \eqref{schurxi}, \eqref{decomposition}, \eqref{defLambda}, \eqref{cvGp}, \eqref{secondterm} and the tightness of $Q_N(z)$ (which readily follows from Lemma \ref{lemFQ}), by using Slutsky's theorem and  classical operations on convergence in probability.
\end{proof} 
We now state an approximation result  in distribution for the finite dimensional distributions of the process $\{Q_N(z), z \in \mathbb C \backslash \mathbb R \}$. 
\begin{lemma} \label{finitedimN}
 Let $(z_1, \ldots, z_q)$ be in $(\mathbb{C}\setminus \mathbb{R})^q$.
 Set $V_N = (Q_N(z_1), \ldots, Q_N(z_q))$.
Then, under the assumptions of Theorem \ref{casnondiag}
$$ d_{LP}(V_N, ({\cal G}_N(z_1), \ldots,{\cal G}_N(z_q))  ) \rightarrow 0$$
where ${\cal G}_N$ is a 
a centered matrix valued Gaussian process   whose distribution is given as follows : \\
 1) the processes $(({\mathcal G}_N)_{ij}(z))_z$ for $1 \leq i \leq j \leq p$ are independent, and $({\mathcal G}_N)_{ji}(z) = \overline{({\mathcal G}_N)_{ij}(\bar{z})}$. \\
 2) For $i \leq p$,
 \begin{eqnarray}
 \lefteqn{\mathbb E(({\cal G}_N)_{ii}(z_k) ({\cal G}_N)_{ii}(z_l))}\\ &= &\frac{(m_4-3 \sigma^4)}{2(N-p)}  \sum_{i=1}^{N-p} ((z_k-\sigma^2 g(z_k) -A_{N-p})^{-1})_{ii} ((z_l-\sigma^2 g(z_l) -A_{N-p})^{-1})_{ii}  \nonumber\\
&&  \qquad \qquad  +
\sigma^4 \int \frac{1}{ (z_k-x)(z_l-x)} d\lambda(x) \label{covarianceN}
\end{eqnarray}
2) For $1 \leq i \not= j \leq p$,
$$ \mathbb E(({\cal G}_N)_{ij}(z_k) ({\cal G}_N)_{ij}(z_l)) = 0$$
and $$
\mathbb E(({\cal G}_N)_{ij}(z_k) \overline{({\cal G}_N)_{ij}(z_l)} ) = \sigma^4  \int \frac{1}{ (z_k-x)(\bar{z}_l-x)} d\lambda(x). $$
\end{lemma}
\begin{proof}
We apply  Proposition  \ref{theoTCLquadratique} to the matrices $B(z) = \hat G_{N-p}(z)$, $z \in \mathbb C \setminus \mathbb{R}$. Note that these  matrices are random but they are independent of the $N\times p$ matrix  $Y$. \\
In order to conclude, 
we need to show that 
\begin{equation} \label{I_N}
 I_N :=  \frac{1}{N-p}\sum_{i=1}^{N-p} \mathbb ((\hat G_{N-p}(z_1))_{ii}(\hat G_{N-p}(z_2))_{ii}) = f_N(z_1, z_2) + o(1) \end{equation}
 with
 $$ f_N(z_1, z_2)  =  \frac{1}{N-p} \sum_{i=1}^{N-p} ((z_k-\sigma^2 g(z_k) -A_{N-p})^{-1})_{ii} ((z_l-\sigma^2 g(z_l) -A_{N-p})^{-1})_{ii} $$
and 
$$ J_N :=\tr_{N-p} (\hat G_{N-p}(z_1)\hat G_{N-p}(z_2) ) \vers^{\mathbb P}_{N\rightarrow \infty}  \int \frac{1}{ (z_1-x)(z_2-x)} d\lambda(x).$$
The second convergence follows from the convergence of $\mu_{{M}_{N-p}}$ towards $\lambda$.\\
For the first one, 
Lemma \ref{estimGij} and \eqref{comparaisonhat} yield that, for $k\leq N-p$,
\begin{eqnarray*}\E\left( (\hat G_{N-p}(z))_{kk}\right)&=&\left[ \left( (z-\sigma^2 \tilde g_{N-p}(z))I_N-A_{N-p}\right)^{-1}\right]_{kk}+ O\left(\frac{1}{\sqrt{N}}\right)\\
&=&\left[\left( (z-\sigma^2  g(z))I_N-A_{N-p}\right)^{-1}\right]_{kk}+o^{(u)}(1)\end{eqnarray*}

Thus, using Lemma \ref{var}, we can deduce that 
\begin{eqnarray}\lefteqn{ \mathbb E((\hat G_{N-p}(z_1))_{kk}(\hat G_{N-p}(z_2))_{kk})}\nonumber\\& =& \left[(z_1-\sigma^2 g(z_1) - A_{N-p})^{-1}\right]_{kk}\left[(z_2-\sigma^2 g(z_2) -  A_{N-p})^{-1}\right]_{kk} \nonumber\\&&+o^{(u)}(1) \label{IN}\end{eqnarray}
 From Lemma \ref{Herbst} in the Appendix, using that $f_N(W)= \frac{1}{N-p} \sum_{i=1}^{N-p} [(z_1-W-A_{N-p})^{-1}]_{ii}[(z_2-W-A_{N-p})^{-1}]_{ii}$
is Lipschitz with constant  $|\Im(z_1)|^{-2}|\Im(z_2)|^{-1} + |\Im(z_1)|^{-1}|\Im(z_2)|^{-2}$, we can deduce that   
 $$  \frac{1}{N-p} \sum_{i=1}^{N-p} (\hat G_{N-p}(z_1))_{ii}(\hat G_{N-p}(z_2) )_{ii}  = \frac{1}{N-p} \sum_{i=1}^{N-p}\mathbb E\left( (\hat G_{N-p}(z_1))_{ii}(\hat G_{N-p}(z_2) )_{ii} \right) + o_{\mathbb P}(1).$$
 We also use to obtain \eqref{covarianceN} from Proposition \ref{theoTCLquadratique} that for $\sigma^2 = 1$,
$$\mathbb E(|y_{11}|^4) - 2 = \frac{1}{2} (m_4 -3),$$
where we recall that $\mu$ is the distribution of $\sqrt{2}\Re{y_{11}}$ and $\sqrt{2}\Im{y_{11}}$. 
\end{proof}

\begin{corollary}\label{finiteTN}
 Under the assumptions of Theorem \ref{casnondiag}, for any $(z_1, \ldots, z_q)$ be in $(\mathbb{C}\setminus \mathbb{R})^q$, 
$$ d_{LP}( (\xi_N(z_1), \ldots,\xi_N(z_q)), ({\cal T}_N(z_1), \ldots,{\cal T}_N(z_q))  ) \rightarrow 0$$ where $\xi_N$ is the process defined by \eqref{defxiN} and   \begin{equation} \label{defiTN}{\cal T}_N(z)=(z-\theta - \sigma^2 g(z))^{-2}\left(P^*(W_p+\mathcal{G}_N(z))P\right)_{11}\end{equation}  $\mathcal{G}_N$ being a  centered Gaussian matrix valued process independent from $W_p$  whose distribution is described in Lemma \ref{finitedimN}.
\end{corollary}

\subsection{Fluctuations of the eigenvector}\label{fluctvect}
Recall from Proposition \ref{backtoresolvent} that, for any $k\in \mathbb{N^*}$,  $\Phi_N$ defined as $$\sqrt{N} (|\langle u_{i_0}, v_{i_0}\rangle|^2 -  \tau_N(\theta))$$ has the representation 
\begin{equation} \label{defPhi_N}
\Phi_N = - \frac{1}{\pi}  \int_{\mathbb C} \bar{\partial} F_k(h)(z) \xi_N(z) d^2z.
\end{equation}
We follow  the proof of  Lemma 6.3 in \cite{NY} based upon  the following estimates, from  \eqref{xi} :
\begin{equation}\label{xiNmajunif} \sup_N\mathbb E(|\xi_N(z)|) \leq P_3(|\Im(z)|^{-1}),\end{equation}
and from Lemma  \ref{finitedimN}  and \eqref{majts}, ($\mathcal{T}_N$ being  defined by \eqref{defiTN})
\begin{equation}\label{TNmaj} \sup_N\mathbb E(|\mathcal{T}_N(z)|) \leq P_5(|\Im(z)|^{-1}),\end{equation}
where  
 $P_3$, $P_4$ and $P_5$ are  some polynomial with nonnegative coefficients.  Hence in the following, in \eqref{defPhi_N}, we choose $k$ greater than  the degrees of $P_3$, $P_4$ and $P_5$.
\begin{proposition}\label{propfinale2}  Under the assumptions of Theorem \ref{casnondiag},
$ d_{LP}(\Phi_N,\tilde \Phi_N ) \rightarrow 0$  where $\tilde \Phi_N$ is given by 
$$\tilde \Phi_N =   (P^*(c_{\theta, \nu} W_{p}  +Z_{p,N}) P)_{11},$$where
$W_p$ is a Wigner matrix of size $p$, $Z_{p,N}$ is a centered Gaussian Hermitian matrix  of size $p$ with independent entries (modulo the symmetry condition); the diagonal coefficients are iid with variance  
 \begin{equation}
 \frac{1}{2}(m_4 - 3 \sigma^4) A_{\theta, \nu, N} + \sigma^4 B_{\theta, \nu}
 \end{equation}
and the off diagonal elements are i.i.d. complex Gaussian with distribution $Z$ such that $\mathbb E(Z^2) = 0$ and $\mathbb E(|Z|^2) =  \sigma^4 B_{\theta, \nu}$. \\
See Eq.\eqref{cov-diag2} for the definitions of $c_{\theta, \sigma}$, $B_{\theta, \nu}$ and \eqref{defAthetaN} for the definition of $A_{\theta, \nu, N}$.
\end{proposition}

\begin{proof} For the reader's convenience, we repeat here the strategy of  Lemma 6.3 in \cite{NY}.
For any  $0<\epsilon<1$  small enough  such that $\chi \equiv 1$ on $]-\epsilon; \epsilon[$,  define  
$$D_\epsilon= \{ z\in \C,  \vert \Im(z)\vert  \geq  \epsilon\}.$$
Set $$U_N=\int_{\mathbb C} \bar{\partial} F_k(h)(z)\xi_N(z)  d^2z,\;
U_N^\epsilon=\int_{D_\epsilon} \bar{\partial} F_k(h)(z)\xi_N(z)  d^2z$$ and
$$V_N =\int_{\mathbb C} \bar{\partial} F_k(h)(z){\cal T}_N(z) d^2z,\;V^\epsilon_N=\int_{D_\epsilon} \bar{\partial} F_k(h)(z){\cal T}_N(z) d^2z,$$
 where $\mathcal{T}_N$ is defined by \eqref{defiTN}. Let $f$ be a bounded continuous complex function on $\C$.\\
We have 
\begin{eqnarray*} \left|\mathbb{E}\left( f(U_N)\right)- \mathbb{E}\left( f(V_N)\right)\right|&\leq& 
\left|\mathbb{E}\left( f(U_N)\right)- \mathbb{E}\left( f(U_N^\epsilon)\right)\right|\\&&+\left|\mathbb{E}\left( f(U_N^\epsilon)\right)- \mathbb{E}\left( f(V_N^\epsilon)\right)\right|\\&&+\left|\mathbb{E}\left( f(V_N^\epsilon)\right)- \mathbb{E}\left( f(V_N)\right)\right|
\end{eqnarray*}
 Let $\delta>0$.\\

\noindent i) For any $\eta>0$ and $K>0$, we have  \begin{eqnarray} \lefteqn{\left|\mathbb{E}\left( f(U_N)\right)- \mathbb{E}\left( f(U_N^\epsilon)\right)\right|}\nonumber\\
&\leq & \left|\mathbb{E}\left( f(U_N)- f(U_N^\epsilon)\right) \1_{\vert U_N-U_N^\epsilon\vert > \eta}\right|\label{I}\\&&+
\left|\mathbb{E}\left( f(U_N)- f(U_N^\epsilon)\right) \1_{\vert U_N-U_N^\epsilon\vert \leq \eta, \vert U_N\vert \vee \vert U_N^\epsilon\vert >K}\right|\label{II}\\&&+\left|\mathbb{E}\left( f(U_N)- f(U_N^\epsilon)\right) \1_{\vert U_N-U_N^\epsilon\vert \leq \eta, \vert U_N\vert \vee \vert U_N^\epsilon\vert \leq K}\right|\label{III}\end{eqnarray}

In the following, the  constant $C>0$ may vary from line to line.
By \eqref{real-axis}, for any $z=x+iy $ in a neighborhood of the real axis,
\begin{equation}\label{majepsilon}|\bar{\partial} F_k(h)(z)| \leq C |y|^k.\end{equation} 
\eqref{barpartial}, \eqref{xiNmajunif} and  \eqref{majepsilon}  readily yield that  for any $\epsilon >0$,
 for any $N$,  $$\mathbb{E}\left( \left|U_N\right| \right)\vee\mathbb{E}\left(  \left|U_N^\epsilon\right|\right) \leq \int_{\C} | \bar{\partial} F_k(h)(z)| \mathbb E|\xi_N(z)|  d^2z \leq  C,$$ and therefore 
\begin{eqnarray*}\mathbb{P}\left( \left|U_N\right| \vee \left|U_N^\epsilon\right|>K\right)&\leq& \mathbb{P}\left( \left|U_N\right|>K\right)+ \mathbb{P}\left( \left|U_N^\epsilon\right|>K\right)\\&\leq& \frac{2}{K} \mathbb{E}\left( \left|U_N\right| \right)\vee\mathbb{E}\left(  \left|U_N^\epsilon\right|\right)\leq \frac{C}{K}.
\end{eqnarray*}
 Thus, we can choose $K$ such that, for any $\epsilon>0$, any $\eta>0$ and any $N$,   the RHS in \eqref{II} is smaller than $\delta$. Now, since $f$ is uniformly continuous on $\{z \in \C\setminus \R,
\vert z \vert \leq K\}$,  one can choose $\eta$ small enough
such that, for any $\epsilon>0$ and any $N$,   \eqref{III} 
are smaller that $\delta$.
Finally, by using  \eqref{xiNmajunif} and  \eqref{majepsilon}, 
  for any $\epsilon$ small enough, for any $N$, 
\begin{equation} \label{surleqepsilon}\mathbb{E} \left|U_N-U_N^{\epsilon}\right|= \mathbb{E} \left|\int_{\{z, |\Im(z)|< \epsilon\}} \bar{\partial} F_k(h)(z)\xi_N(z)  d^2z\right|\leq C\epsilon,\end{equation} and then $$\mathbb{P}\left( \left|U_N-U_N^\epsilon\right|>\eta\right)\leq \frac{C}{\eta} \epsilon.$$
The term $|\mathbb{E}\left( f(V_N^\epsilon)\right)- \mathbb{E}\left( f(V_N)\right)|$ is treated in the same way, using \eqref{TNmaj}.

\noindent ii)
Note that  since h and $\chi$ are compactly supported,  $\int_{ D_\epsilon} \bar{\partial} F_k(h)(z)\xi_N(z) d^2z$ may be seen as  an integral on a fixed compact set $K_\epsilon\subset D_\epsilon$. Proposition \ref{tensiondexiN} readily yields that $\xi_{N}$ is tight on the set ${\cal C}( K_\epsilon)$ of complex continuous functions on $K_\epsilon$.
Thus, using Corollary \ref{finiteTN} ,  since 
$f \mapsto \int_{ D_\epsilon} \bar{\partial} F_k(h)(z)f(z)  d^2z$ is continuous on $\mathcal{C}(K_\epsilon)$, it remains to  prove the tightness of the  process $\{{\cal T}_N (.) \}$ on any compact set $K$ in $\{z, \vert \Im z \vert  \geq \epsilon\}$ to deduce from  Lemma 5.7 in \cite{NY} that :

$$|\mathbb{E}\left( f(U_N^\epsilon)\right)- \mathbb{E}\left( f(V_N^\epsilon)\right)| \rightarrow_{N\rightarrow +\infty} 0.$$
\\
We postpone the proof of the tightness of the  process $\{{\cal T}_N (.) \}$ (see Lemma \ref{tightTN} below)
Thus, up to the proof of Lemma \ref{tightTN}, the convergence to 0 of $d_{LP}(\Phi_n, - \frac{1}{\pi} \int_{\mathbb C} \bar{\partial} F_k(h)(z){\cal T}_N(z) d^2z )$   follows.

\noindent
It remains to compute $ c_{\theta, \nu},  B_{\theta, \nu} ,  A_{\theta,\nu, N}$. \\
The computation of $c_{\theta, \nu}$ follows from Proposition \ref{prop-calculHS} :
$$ c_{\theta, \nu} =  -\frac{1}{\pi}  \int_{\mathbb C} \bar{\partial} F_k(h)(z) \frac{1}{(z-\sigma^2 g(z) - \theta)^2} dz = {\rm Res}(  \frac{1}{(z-\sigma^2 g(z) - \theta)^2}, \rho_\theta).$$
A straightforward computation gives 
$$ {\rm Res}(  \frac{1}{(z-\sigma^2 g(z) - \theta)^2}, \rho_\theta) = - \frac{\omega''(\rho_\theta)}{(\omega'(\rho_\theta))^3} = H''(\theta)$$
where $\omega(z) = z - \sigma^2g(z)$ ($g := g_{\lambda}$) and $H(z) = z + \sigma^2 g_\nu(z)$. \\
From  Lemma \ref{finitedimN} and using Fubini theorem, the diagonal entries of $Z_N$ have variance equal to 
$$  \frac{1}{2} (m_4 - 3 \sigma^4) A_{\theta, \nu, N}  + \sigma^4 B_{\theta, \nu}$$ and the off diagonal entries of $Z_N$ have variance equal to
$\sigma^4 B_{\theta, \nu}$
where 
\begin{eqnarray*}
A_{\theta, \nu,N} &=&\frac{1}{N-p} \sum_{i=1}^{N- p} \left(  \frac{1}{\pi}  \int_{\mathbb C} \bar{\partial} F_k(h)(z) \frac{[(z-\sigma^2 g(z) -A_{N-p})^{-1}]_{ii}}{(z-\sigma^2 g(z) - \theta)^2} d^2z\right)^2 d\nu(x) \\
B_{\theta, \nu} &=& \int_\mathbb R \left(  \frac{1}{\pi}  \int_{\mathbb C} \bar{\partial} F_k(h)(z) \frac{1}{(z-\sigma^2 g(z) - \theta)^2(z -x)} d^2z\right)^2 d\lambda(x).
\end{eqnarray*}
The functions $\phi_i(z) =   \frac{[(z-\sigma^2 g(z) -A_{N-p})^{-1}]_{ii}}{(z-\sigma^2 g(z) - \theta)^2}$  for $i \leq N-p$ and $\phi_x(z) =  \frac{1}{(z-\sigma^2 g(z) - \theta)^2(z -x)}$ for $x \in {\rm supp}(\lambda)$ satisfy the hypothesis of Proposition \ref{prop-calculHS}.
Straightforward computations lead to : 
\begin{eqnarray*}
A_{\theta, \nu,N} &= &\frac{1}{N-p} \sum_{i=1}^{N- p}  ({\rm Res}(\frac{[(z-\sigma^2 g(z) -A_{N-p})^{-1}]_{ii}}{(z-\sigma^2 g(z) - \theta)^2}, \rho_\theta))^2 \\
&=&  \frac{1}{N-p} \sum_{i=1}^{N-p} \left( H''(\theta) [(\theta I_{N-p} - A_{N-p})^{-1}]_{ii} - H'(\theta) [(\theta I_{N-p} - A_{N-p})^{-2}]_{ii} \right)^2 \\
&=&  \frac{1}{N-p} \sum_{i=1}^{N-p} \left( \sigma^2 g''_\nu(\theta) [(\theta I_{N-p} - A_{N-p})^{-1}]_{ii} - (1+ \sigma^2 g'_\nu(\theta)) [(\theta I_{N-p} - A_{N-p})^{-2}]_{ii} \right)^2
\end{eqnarray*}
and
\begin{eqnarray*}
B_{\theta, \nu} &= & \int_\mathbb R ({\rm Res}(\frac{1}{(z-\sigma^2 g(z) - \theta)^2(z -x)}, \rho_\theta))^2 d\lambda(x) \\
&=&  -\frac{1}{6}{g_\nu'''(\theta)} -\frac{\sigma^2}{2}({g_\nu''(\theta)})^2 \frac{1+2\sigma^2 g'_\nu(\theta)}{1+\sigma^2 g'_\nu(\theta)}.
\end{eqnarray*}

\end{proof}
We now prove the following Lemma, used in the proof of the above Proposition.
 \begin{lemma} 
 \label{tightTN} Let $K$ be a compact subset in $\{z, \vert \Im z \vert  \geq \epsilon\}$, for some $\epsilon >0$.
 The  process $\{{\cal T}_N (.) \}$ defined in \eqref{defiTN} is a tight sequence on $K$, more precisely, 
 \begin{equation} \label{Kol}
 \sup_{z_1, z_2 \in K, n \in \mathbb N} \frac{ \mathbb E(|{\cal T}_N(z_1) - {\cal T}_N(z_2)|^2)}{ |z_1 -z_2|^2} < \infty.
 \end{equation}
 \end{lemma}
 \begin{proof} From Lemma \ref{finitedimN}, 
 \begin{eqnarray*}
\lefteqn{  \mathbb E(|({\cal G}_N)_{ij}(z_1) - ({\cal G}_N)_{ij}(z_2)|^2)}\\& = & \delta_{ij}\frac{1}{2} (m_4-3 \sigma^4) \frac{1}{N-p} \sum_{i=1}^{N-p}(|((z_1-\sigma^2 g(z_1) -A_{N-p})^{-1})_{ii} -((z_2-\sigma^2 g(z_2) -A_{N-p})^{-1})_{ii}|^2) \\
  && \quad+\sigma^4 \int |\frac{1}{z_1-x} - \frac{1}{z_2-x} |^2 d\lambda(x).
\end{eqnarray*}

From the resolvent identity,
\begin{eqnarray*} \lefteqn{(z_1-\sigma^2 g(z_1) -A_{N-p})^{-1} -(z_2-\sigma^2 g(z_2) -A_{N-p})^{-1}}\\& = (z_2 - z_1-\sigma^2 (g(z_2)-g(z_1)) ) (z_1-\sigma^2 g(z_1) -A_{N-p})^{-1} (z_2-\sigma^2 g(z_2) -A_{N-p})^{-1},\end{eqnarray*}
thus,
\begin{eqnarray*}
|((z_1-\sigma^2 g(z_1) -A_{N-p})^{-1})_{ii} -((z_2-\sigma^2 g(z_2) -A_{N-p})^{-1})_{ii}|
&\leq &\frac{1}{ \varepsilon^2} (1+\frac{\sigma^2}{\epsilon^2}) |z_1-z_2|
\end{eqnarray*}
and thus,
$$  \frac{1}{N-p} \sum_{i=1}^{N-p} (|((z_1-\sigma^2 g(z_1) -A_{N-p})^{-1})_{ii} -((z_2-\sigma^2 g(z_2) -A_{N-p})^{-1})_{ii}|^2) \leq \frac{1}{ \varepsilon^4}   (1+\frac{\sigma^2}{\epsilon^2})^2|z_1-z_2|^2.$$
Since moreover
$$ \int |\frac{1}{z_1-x} - \frac{1}{z_2-x} |^2 d\lambda(x) \leq \vert \Im z_1\vert^2 \vert z_2 \vert^2 |z_1-z_2|^2 \leq \frac{1}{ \varepsilon^4}  |z_1-z_2|^2$$
\eqref{Kol} readily follows. The tightness follows from Kolmogorov's  criterion (see \cite{Bi}). \end{proof}

Proposition \ref{propfinale2}  and Proposition \ref{backtoresolvent} readily yield Theorem \ref{casnondiag}.


\subsection{ Proof of Theorem \ref{propcasnondiag2}}
Theorem \ref{propcasnondiag2} follows from Theorem  \ref{casnondiag} once we proved that $A_{\theta,\nu, N}$ converge to $A_{\theta, \nu}$.
\begin{lemma}
Assume that the matrix $A_N$ satisfies {\bf (A')} with $A_{N-p}$ diagonal. Then, the sequence  $(A_{\theta,\nu, N})_N$ defined by \eqref{defAthetaN} converges to  $A_{\theta,\nu}$ defined by \eqref{cov-diag2}.
\end{lemma}
\begin{proof}
Denote by $d_i$ the eigenvalues of $A_{N-p}$.
 \begin{eqnarray*}
 A_{\theta,\nu, N}&=&  \frac{1}{N-p} \sum_{i=1}^{N-p} \left( H''(\theta) (\theta-d_i)^{-1} - H'(\theta)  (\theta-d_i)^{-2} \right)^2 \\
 & \longrightarrow_{N \rightarrow \infty}&  (H''(\theta))^2 \int \frac{1}{(\theta-x)^2} d\nu(x) - 2 H'(\theta) H''(\theta) \int \frac{1}{(\theta-x)^3} d\nu(x) \\
 && \qquad + (H'(\theta))^2  \int \frac{1}{(\theta-x)^4} d\nu(x) \\
 &=& -(H''(\theta))^2 g'_{\nu}(\theta) - H'(\theta) H''(\theta) g''_{\nu}(\theta) - \frac{1}{6} (H'(\theta))^2 g'''_{\nu}(\theta)
 \end{eqnarray*}
 Using, $H'(\theta) = 1 + \sigma^2 g'_{\nu}(\theta)$ and $H''(\theta) = \sigma^2 g''_{\nu}(\theta)$, we obtain the formula for $A_{\theta, \nu}$ given in \eqref{cov-diag2}.

\end{proof}

\section{Appendix: Poincar\'e inequality and concentration phenomenon}
A probabilty $\mu$ satisfies a Poincar\'e inequality if for any
${\cal C}^1$ function $f: \mathbb{R}\rightarrow \mathbb{C}$  such that $f$ and
$f'$ are in $L^2(\mu)$,
$$\mathbf{V}(f)\leq C_{PI}\int  \vert f' \vert^2 d\mu ,$$
\noindent with $\mathbf{V}(f) = \int \vert
f-\int f d\mu \vert^2 d\mu$. \\
 If the law of a random variable $X$ satisfies the Poincar\'e inequality with constant $C_{PI}$ then, for any fixed $\alpha \neq 0$, the law of $\alpha X$ satisfies the Poincar\'e inequality with constant $\alpha^2 C_{PI}$.\\
Assume that  probability measures $\mu_1,\ldots,\mu_M$ on $\mathbb{R}$ satisfy the Poincar\'e inequality with constant $C_{PI}(1),\ldots,C_{PI}(M)$ respectively. Then the product measure $\mu_1\otimes \cdots \otimes \mu_M$ on $\mathbb{R}^M$ satisfies the Poincar\'e inequality with constant $\displaystyle{C_{PI}^*=\max_{i\in\{1,\ldots,M\}}C_{PI}(i)}$ in the sense that for any differentiable function $f$ such that $f$ and its gradient ${\rm grad} f$ are in $L^2(\mu_1\otimes \cdots \otimes \mu_M)$,
$$\mathbf{V}(f)\leq C_{PI}^* \int \Vert {\rm grad} f \Vert_2 ^2 d\mu_1\otimes \cdots \otimes \mu_M$$
\noindent with $\mathbf{V}(f) = \int \vert
f-\int f d\mu_1\otimes \cdots \otimes \mu_M \vert^2 d\mu_1\otimes \cdots \otimes \mu_M$.
\begin{lemma}\label{Herbst}{ Lemma 4.4.3 and Exercise 4.4.5 in \cite{AGZ} or Chapter 3 in  \cite{L}. }
Let $\mathbb{P}$ be a probability measure on $\mathbb{R^M}$ which satisfies a Poincar\'e inequality with constant $C_{PI}$. Then there exists $K_1>0$ and $K_2>0$ such that,  for any  Lipschitz function $F$  on $\mathbb{R}^M$ with Lipschitz constant $\vert F \vert_{Lip}$,
$$\forall \epsilon> 0,  \, \mathbb{P}\left( \vert F-\mathbb{E}_{\mathbb{P}}(F) \vert > \epsilon \right) \leq K_1 \exp\left(-\frac{\epsilon}{K_2 \sqrt{C_{PI}} \vert F \vert_{Lip}}\right).$$
\end{lemma}

\section*{Acknowledgements} The authors want to thank an anonymous referee who  encouraged them to establish more general   results and to provide a better readability, which led to an overall improvement of the paper.



\end{document}